\documentclass[journal,twoside,web]{ieeecolor}
\usepackage{generic}
\usepackage{cite}
\usepackage{amsmath,amssymb,amsfonts}
\usepackage{graphicx}
\usepackage{algorithm,algorithmic}
\usepackage{hyperref}
\usepackage{textcomp}
\usepackage{dsfont}
\usepackage{enumerate}
\usepackage{subfig}
\usepackage[markup=default,commandnameprefix=always,authormarkup=none]{changes}

\newtheorem{theorem}{Theorem}
\newtheorem{lemma}{Lemma}
\newtheorem{corollary}{Corollary}
\newtheorem{definition}{Definition}
\newtheorem{assumption}{Assumption}
\newtheorem{remark}{Remark}

\def\BibTeX{{\rm B\kern-.05em{\sc i\kern-.025em b}\kern-.08em
    T\kern-.1667em\lower.7ex\hbox{E}\kern-.125emX}}
\begin{document}
\title{Sinkhorn Distributionally Robust State Estimation via System Level Synthesis}

\author{Yulin Feng, Xianyu Li, Steven X. Ding, Hao Ye, and Chao Shang, \IEEEmembership{Member, IEEE}
\thanks{This work was supported by the National Natural Science Foundation of China (Grant No. 62373211), and the Open Research Project of the State Key Laboratory of Industrial Control Technology, China (Grant No. ICT2025B10) \textit{(Corresponding author: Chao Shang)}.}
\thanks{Yulin Feng, Xianyu Li, Hao Ye and Chao Shang are with the Department of Automation, Beijing National Research Center for Information Science and Technology, Tsinghua University, Beijing 100084, China (e-mail: fyl23@mails.tsinghua.edu.cn;  li-xy22@mails.tsinghua.edu.cn; haoye@tsinghua.edu.cn; c-shang@tsinghua.edu.cn).}
\thanks{Steven X. Ding is with the Institute for Automatic Control and Complex Systems, University of Duisburg-Essen, 47057 Duisburg, Germany (e-mail: steven.ding@uni-due.de).}}

\maketitle

\begin{abstract}
In state estimation tasks, the usual assumption of an exactly known disturbance distribution is often unrealistic and renders the estimator fragile in practice. The recently emerging Wasserstein distributionally robust state estimation (DRSE) design can partially mitigate this fragility; however, its worst-case distribution is provably discrete, which deviates from the inherent continuity of real-world distributions and results in over-pessimism. In this work, we develop a new Sinkhorn DRSE design within the system level synthesis scheme with the aim of shaping the closed-loop errors under the unknown continuous disturbance distribution. For uncertainty description, we adopt the Sinkhorn ambiguity set that includes an entropic regularizer to penalize non-smooth and discrete distributions within a Wasserstein ball. We present the first result of finite-sample probabilistic guarantee of the Sinkhorn ambiguity set. Then we analyze the limiting properties of our Sinkhorn DRSE design, thereby highlighting its close connection with the generic $\mathcal{H}_2$ design and Wasserstein DRSE. To tackle the min-max optimization problem, we reformulate it as a finite-dimensional convex program through duality theory. By identifying a compact subset of the feasible set guaranteed to enclose the global optimum, we develop a tailored Frank-Wolfe solution algorithm and formally establish its convergence rate. The advantage of Sinkhorn DRSE over existing design schemes is verified through numerical case studies.

\end{abstract}


\section{Introduction}
State estimation is to infer the internal states of dynamic systems from noisy measurements, which has played a key role in predictive control \cite{hakobyan2024wasserstein}, fault diagnosis \cite{ding2014data}, and signal processing \cite{kim2005qrd}. The well-known Kalman filter minimizes the mean squared error (MSE) via a recursive solution in closed form suitable for online implementation \cite{kalman1960new, chen2011kalman,anderson2005optimal}. It has been proven to provide the optimal unbiased estimate under the assumption of exactly known system dynamics and Gaussian noise. Over the past few decades, this kind of least square $\mathcal{H}_2$ design has witnessed significant developments \cite{barbosa2005robust,duan2006robust,talebi2019distributed}. However, uncertainty arising from unmodeled dynamics and non-Gaussian disturbances can often significantly degrade the performance of $\mathcal{H}_2$ filters. To tackle the challenge of designing robust estimators under unknown probability distributions and deterministic disturbances, the distribution-free $\mathcal{H}_{\infty}$ filters formulate a min-max problem under the assumption of bounded uncertainty \cite{poor1981minimax, zames2003feedback}. However, the $\mathcal{H}_{\infty}$ approach may result in over-conservative estimation due to the lack of statistical information. Unlike the aforementioned single-step filters, a distribution-free alternative is to design an observer, which can asymptotically stabilize the error dynamics and offers distinct advantages when dealing with non-Gaussian disturbances. Recent advances in system level synthesis (SLS) provide a rigorous pathway for observer and control design by shaping the closed-loop response under unknown disturbances \cite{anderson2019system,brouillon2023minimal,brouillon2025distributionally}. By formulating the design as a convex optimization problem, the SLS framework effectively handles complicated disturbances while accommodating hard constraints.

To address the aforementioned issues, the distributionally robust optimization (DRO) framework offers a new paradigm, which bridges the gap between $\mathcal{H}_2$ and $\mathcal{H}_{\infty}$ approaches, and thus strikes a balance between robustness and performance \cite{feng2026distributionally,shang2022generalized}. Within this scheme, the unknown disturbance distribution is captured by constructing a so-called ambiguity set, composed of some candidate distributions sharing similar statistical properties. This new uncertainty characterization has attracted attention in state estimation, where the worst-case estimation error is minimized to hedge against distributional ambiguity. The earliest literature can be traced back to \cite{levy2012robust}, where a Kalman filter considering distribution mismatch via $\tau$-divergence was proposed. Building upon the Gaussian assumption, a series of studies on distributionally robust state estimation (DRSE) leverage Wasserstein distance \cite{shafieezadeh2018wasserstein}, moment information \cite{wang2021distributionally}, martingale constraint \cite{lotidis2023wasserstein} and Kullback-Leibler (KL) divergence \cite{xu2024globalized} to handle the distributional shift and outliers in the estimator design. However, the assumption of Gaussianity is fragile in practice. To address non-Gaussian scenarios, the Wasserstein DRO provides a powerful tool due to its out-of-sample guarantees, clear interpretability and computational tractability. In \cite{shafieezadeh2019regularization}, a DRSE method was developed using Wasserstein ambiguity set centered at an empirical distribution, and its equivalence with a regularization scheme was established. Focusing on the temporally global estimation error, an infinite-horizon Wasserstein DRSE and its rational $\mathcal{H}_{\infty}$ approximation were proposed in \cite{hajar2025distributionally}. 

Despite the guaranteed robustness of Wasserstein DRSE designs, they suffer from over-conservatism. In \cite{gao2023distributionally}, the worst-case distribution of Wasserstein DRO was proven to be pathologically discrete, which may deviate from the ground truth. Indeed, the uncertainty in real-world dynamical processes is typically governed by continuous probability distributions \cite{shang2022generalized}. This hints that the Wasserstein ambiguity set contains an excessive number of unrealistic distributions. Recently, the Sinkhorn distance, an entropy regularized variant of Wasserstein distance, has been adopted to account for those smooth and continuous distributions, which helps to reduce the conservatism of Wasserstein DRO \cite{feng2018model,wang2025sinkhorn}, and has found applications in a variety of decision-making tasks \cite{cescon2025data,ou2026sinkhorn}.

In this paper, a new Sinkhorn DRSE approach is proposed. The main contributions are as follows:
\begin{itemize}
    \item We formalize the DRSE problem within the SLS framework using the Sinkhorn ambiguity set, which introduces an entropic regularizer to reduce the conservatism of Wasserstein DRSE. For the Sinkhorn ambiguity set, {we establish a formal finite-sample probabilistic guarantee that is new in literature. }
     \item By adapting known results in Sinkhorn distributionally robust control \cite{cescon2025data}, we reformulate the min-max design problem as a finite-dimensional convex program. To solve this nonlinear program, a tailored first-order Frank-Wolfe algorithm is developed as a more efficient solution strategy than the bisection-search method proposed in \cite{wang2025sinkhorn}. Crucially, to account for the unbounded feasible region, we identify a compact subset thereof that is guaranteed to include the global optimal solution such that the Frank-Wolfe algorithm becomes applicable. Notably, the convergence rate of the first-order solution algorithm is formally established. 
     \item The performance of the proposed Sinkhorn DRSE and the effectiveness of the Frank-Wolfe solution algorithm are illustrated through numerical experiments. The comprehensive results show that the Sinkhorn DRSE improves estimation accuracy, particularly in small-sample regimes, while the proposed Frank-Wolfe algorithm is less demanding time than the bisection-search method.
\end{itemize}

The rest of this article unfolds as follows.
The preliminaries of the SLS framework and the standard state estimation design problem are revisited in Section \ref{preli}. Section \ref{mainr} presents the formal formulation of the Sinkhorn DRSE design problem, its reformulation and tailored solution algorithm. The case study results are given in Section \ref{case}. Section \ref{concl} concludes this article. 

\textit{Notations:} We use $\mathbb{S}^n$ ($\mathbb{S}^n_+$) to denote the space of (positive semi-definite) symmetric matrices in $\mathbb{R}^{n \times n}$, and $\mathbb{N}_{i:j}$ to denote the consecutive integers $\{i,\cdots,j\}$. ${0}_{ n\times m}\in\mathbb{R}^{n\times m}$, $I_n\in\mathbb{R}^{n\times n}$ and ${1}_n \in \mathbb{R}^n$ denote a zero matrix, an identity matrix and an all-ones vector, whose subscript may be dropped when no confusion arises. For a random variable $\xi$ supported on a measurable set $\Xi$, the sets of measures and probability measures on $\Xi$ are denoted by $\mathcal{M}(\Xi)$ and $\mathcal{P}(\Xi)$. Given two measures $\alpha,\beta$, we write $\alpha\ll \beta$ to imply the absolute continuity of ${\alpha}$ with respect to $\beta$, and $\alpha \odot\beta$ as the product measure of $\alpha$ and $\beta$. $\mathbb{E}_{\xi\sim\mathbb{P}}[f(\xi)]$ denotes the expectation of the function $f(\xi)$ with respect to the probability measure$\ mathbb{P}$. $\mathcal{N}(\mu,\Sigma)$ and $\mathcal{L}(\mu,\Sigma)$ denote the Gaussian distribution and Laplace distribution with mean $\mu$ and covariance $\Sigma$, respectively. $\delta_{\hat{\xi}}$ denotes the Dirac distribution at $\hat{\xi}$,  and $\mathcal{U}(a,b)$ denotes the uniform distribution on $[a,b]$. For a vector $x$, $x^\top$ and $\lVert x\rVert=\sqrt{x^\top x}$ denote its transpose and Euclidean norm. For a matrix $P$, ${\rm vec}\{P\}$ denotes the vectorization of $P$, ${\rm Tr}\{P\}$ denotes its trace, $P^\top$ denotes its transpose, $P^\dagger$ denotes its Moore-Penrose inverse,  $\bar \sigma(P)$ denotes its maximum eigenvalue, $\lVert P\rVert_{F} $ denotes its Frobenius norm and $\lVert P\rVert_{2} $ denotes its spectral norm.  ${\rm diag}\{P_1,\cdots,P_n\}$ denotes the block diagonal matrix with diagonal block matrices $P_1,\cdots,P_n$, $\langle P_1,P_2\rangle={\rm Tr}\{P_1^\top P_2\}$ denotes the Frobenius inner product, and $ \otimes  $ denotes the Kronecker product operator. For a symmetric $P$, $P \succeq(\succ) 0$ indicates the positive (semi-)definiteness of $P$. For a set $\mathcal{S} \ni (x,y)$, we define its projection to $x$ as ${\rm proj}_{x}(\mathcal{S})
=
\left\{x \,\middle|\, \exists\,y\text{ such that }(x,y)\in\mathcal{S}\right\}$ and its diameter under an arbitrary norm $\lVert\cdot\rVert$ as ${\rm diam}_{\lVert\cdot\rVert}(\mathcal{S})=
\sup_{\substack{(x_1,y_1)\in\mathcal{S}\\(x_2,y_2)\in\mathcal{S}}}
\left\lVert
\begin{bmatrix}x_1\\y_1\end{bmatrix}
-
\begin{bmatrix}x_2\\y_2\end{bmatrix}
\right\rVert$.

\section{Preliminaries} \label{preli}
In this section, we briefly review the SLS framework and the standard $\mathcal{H}_2$ state estimation design problem.
We consider the state estimation problem of the following linear stochastic time-varying system:
\begin{equation}\begin{cases} 
{x}(t+1)=A_t{x}(t)+B_t {w}(t)\\
{y}(t)=C_t {x}(t)+D_t{w}(t)&\end{cases} \label{eq_statepace}
\end{equation}
where ${x}(t)\in\mathbb{R}^{n_{x}}$ is the unobserved state, ${y}(t)\in\mathbb{R}^{n_y}$ is the measurable output, and $w(t)\in \mathbb{R}^{n_w}$ is the unknown disturbance vector. $A_t \in \mathbb{R}^{n_x \times n_x}$, $B_t \in \mathbb{R}^{n_x \times n_w}$, $C_t \in \mathbb{R}^{n_y \times n_x}$ and $D_t \in \mathbb{R}^{n_y \times n_w}$ are exactly known time-varying state-space matrices at time index $t$. 

\begin{remark}
 The disturbance $w(t)$ in \eqref{eq_statepace} provides a unified representation of uncertainty from multiple sources. Our system formulation \eqref{eq_statepace} using a single disturbance vector $w(t)$ is without loss of generality. For example, by setting $w(t)=\begin{bmatrix}v_{\rm p}(t)^\top & v_{\rm m}(t)^\top\end{bmatrix}^\top$, $B_t=\begin{bmatrix}I & 0\end{bmatrix}$, and $D_t=\begin{bmatrix}0 & I\end{bmatrix}$, \eqref{eq_statepace} can represent the presence of process noise $v_{\rm p}(t)$ and measurement noise $v_{\rm m}(t)$ governed by different distributions.
\end{remark}

\begin{assumption}[\cite{liu2025uniform}]
    The system \eqref{eq_statepace} is uniformly completely observable for any time index $t$. 
\end{assumption}

To estimate the state ${x}(t)$ within the forecasting horizon $t\in[t_0, t_0+T]$ with $T\geq n_x$, the following observer is implemented:
\begin{equation}
    \hat{x}(t+1) =A_t\hat{x}(t) +\sum^{t}_{i=t_0}L_{i| t}\left [ y(i)-C_i  \hat{x}(i) \right ], \label{eq_observer}
\end{equation}
where $L_{i| t}\in \mathbb{R}^{n_x\times n_y}$ represents the observer gain matrices corresponding to the output error $y(i)-C_i  \hat{x}(i)$ at time $i \le t$. This system-level parameterization of linear feedback policies is expressive and includes the well-known Luenberger observer as a special case. While the latter is a recursive estimator with reliance on instantaneous corrections and is thus sensitive to outliers, \eqref{eq_observer} offers enhanced flexibility and robustness when handling unknown complicated disturbances.

\begin{remark} \label{rmk_0}
It is assumed without loss of generality that the disturbance $w(t)$ is zero-mean and there is no extra input. This is because the observer \eqref{eq_observer} can be easily generalized as $   \hat{x}(t+1) =A_t\hat{x}(t)+B_t\mu(t)+\bar B_tu(t) +\sum^{t}_{i=t_0}L_{i| t}[y(i)-C_i  \hat{x}(i)-D_i\mu(i)]$ to handle the case with non-zero disturbance mean $\mathbb{E}\left[w\left(t\right)\right]=\mu(t) \neq0$ or exactly known inputs $u(t)$. 
\end{remark}

The observer gains must be carefully designed to stabilize the dynamics of the estimation error, defined as $e(t) \triangleq  {x}(t)- \hat{x}(t) $. The error dynamics of the observer \eqref{eq_observer} are derived as follows:
\begin{equation*}
\begin{aligned}
  e(t+1) &={x}(t+1)-\hat{x}(t+1)\\
  &=A_t x(t)+B_t w(t)- A_t\hat x(t)\\
  &~~~~~-\sum^{t}_{i=t_0}L_{i|t}\big[C_ix(i)+D_iw(i) -C_i  \hat{x}(i)\big]\\
  &= A_te(t)+B_t w(t) -\sum^{t}_{i=t_0}L_{i|t}\big[C_ie(i)+D_iw(i) \big].
\end{aligned}
\end{equation*}
Consequently, the dynamics of the stacked error $\bar {e}$ is governed by the following compact matrix equation \cite{anderson2019system,brouillon2023minimal}:
 \begin{equation}
     \bar {e}= Z\bar A\bar {e}+\bar B\xi- L\bar C Z\bar e-L\bar D\xi \label{eq_stack_relation}
 \end{equation}
with the stacked vectors defined as
\begin{align*}
&\bar {e} \triangleq \begin{bmatrix}
    e^\top(t_0)&e^\top(t_0+1 )&\cdots&e^\top(t_0+T)
\end{bmatrix}^\top,\\
&\xi\triangleq \begin{bmatrix}
    e^\top(t_0)&w^\top(t_0 )&w^\top(t_0+1 )&\cdots&w^\top(t_0+T-1)
\end{bmatrix}^\top.
\end{align*}
Here, $\bar{e}$ is the trajectory of estimation errors over the forecasting horizon and $\xi$ captures all system uncertainties including the initial estimation error $e(t_0)$ and disturbances $\left\{w(i)\right\}_{i=t_0}^{t_0+T-1}$. The coefficient matrices in \eqref{eq_stack_relation} are defined as:
\begin{align*}
    &Z = \begin{bmatrix}
        0_{n_x\times n_x}&& &\\
        I_{n_x}&\ddots\\
        & \ddots&\ddots &\\
       && I_{n_x} & 0_{n_x\times n_x}
    \end{bmatrix},\\
    &\bar A= {\rm diag}\left\{ A_{t_0},A_{t_0+1},\cdots,A_{t_0+T-1},0_{n_x \times n_x}\right\},\\
    &\bar B =  {\rm diag}\left\{ I_{n_x},B_{t_0},B_{t_0+1},\cdots,B_{t_0+T-1}\right\},\\
    &\bar C =  {\rm diag}\left\{ 0_{n_y\times n_x},C_{t_0},C_{t_0+1},\cdots,C_{t_0+T-1}\right\},\\
    &\bar D = {\rm diag}\left\{ 0_{n_y\times n_x}, D_{t_0},D_{t_0+1},\cdots,D_{t_0+T-1}\right\},\\
    &{\small
L = \begin{bmatrix}
    0_{n_x\times n_y} & 0 & \cdots & 0 \\ 
    0 & L_{t_0| t_0} & \cdots & 0 \\
    \vdots & \vdots & \ddots & \vdots \\
    0 & L_{t_0| t_0+T-1} & \cdots & L_{t_0+T-1| t_0+T-1}
\end{bmatrix}.
}
\end{align*}
The causality of system dynamics is inherently preserved by the block-downshift operator $Z \in \mathbb{R}^{(T+1)n_x\times(T+1)n_x}$. The well-known SLS parameterization then follows from \eqref{eq_stack_relation} \cite{anderson2019system,brouillon2023minimal}:
\begin{equation}
    \bar e := \bar e(L,\xi)  = \Phi_x \bar B \xi- \Phi_y \bar D \xi ,\label{eq_SLS}
\end{equation}
where
\begin{equation}
    \Phi_x = (I-Z \bar A+L\bar CZ )^{-1} \in \mathbb{R}^{(T+1)n_x \times(T+1)n_x } \label{eq_Phi1}
\end{equation}
characterizes the propagation of process disturbances onto the state estimation error, and
\begin{equation}
\Phi_y =  (I-Z \bar A+L\bar CZ )^{-1} L \in \mathbb{R}^{(T+1)n_x \times(T+1)n_y  }\label{eq_Phi2}
\end{equation}
delineates the impact of output measurement noise. Due to the strictly block lower-triangular structure of $Z$, $I-Z \bar A+L\bar CZ $ is invertible, and both $\Phi_x$ and $\Phi_y$ are block lower-triangular:
\begin{align*}
    \Phi_x = \begin{bmatrix}
        \Phi_x^{t_0| t_0}&0 &\cdots&0 \\
        \Phi_x^{t_0| t_0+1}& \Phi_x^{t_0+1| t_0+1}&\ddots&0\\
        \vdots&\vdots&\ddots&\vdots\\
        \Phi_x^{t_0| t_0+T}&\Phi_x^{t_0+1| t_0+T}&\cdots&\Phi_x^{t_0+T| t_0+T}
    \end{bmatrix},\\
        \Phi_y = \begin{bmatrix}
        \Phi_y^{t_0| t_0}&0 &\cdots&0 \\
        \Phi_y^{t_0| t_0+1}& \Phi_y^{t_0+1| t_0+1}&\ddots&0\\
        \vdots&\vdots&\ddots&\vdots\\
        \Phi_y^{t_0| t_0+T}&\Phi_y^{t_0+1| t_0+T}&\cdots&\Phi_y^{t_0+T| t_0+T}
    \end{bmatrix}.
\end{align*}
This structural property enforces causality, ensuring that the estimation error $e(t)$ depends solely on the initial state estimation error of $x(t_0)$ and past disturbances $w(i)$ for $i\in\mathbb{N}_{t_0:t-1}$. In addition, designing a causal observer gain matrix $L$ is equivalent to the design of block lower-triangular maps $\Phi_x$ and $\Phi_y$, if and only if the following \textit{achievability} condition is satisfied \cite[Theorem 2]{brouillon2023minimal}:
\begin{equation}
    \begin{bmatrix}
        \Phi_x & \Phi_y
    \end{bmatrix}    
    \begin{bmatrix}
      I-Z\bar{A} \\  \bar{C}Z 
    \end{bmatrix} = I.
    \label{eq_sls_constraint}
\end{equation} 
Then, the observer gain $L$ can be uniquely determined by:
\begin{equation*}
    L = \Phi^{-1}_x \Phi_y.
\end{equation*}
To account for unknown stochastic disturbance $w$, we adopt the MSE $\mathbb{E}_{\xi \sim \mathbb{P}^{\star}} \left[ \lVert \bar{e}(L,\xi) \rVert^2 \right]$ for performance evaluation, where $\mathbb{P}^{\star}$ stands for the true distribution of $\xi$. Consequently, the state estimator design problem under the SLS framework is formulated as \cite{brouillon2023minimal}:
\begin{subequations}
\begin{align}
    \min_{\Phi_x, \Phi_y} 
    &~ \mathbb{E}_{\xi \sim \mathbb{P}^{\star}} \left[ \lVert \bar{e} \rVert^2 \right] \label{MHE_MSE}
    \\
   {\rm  s.t.}~ &~  \bar e = \Phi_x \bar B \xi- \Phi_y \bar D \xi\\
   &~  \begin{bmatrix}
        \Phi_x & \Phi_y
    \end{bmatrix}    \begin{bmatrix}
  I-Z\bar A  \\     \bar CZ 
    \end{bmatrix}=I \label{MHE_structure} \\
   &~ \Phi_x \text{ and }\Phi_y \text{ are block lower-triangular} \label{MHE_lower_triangle}
    \end{align}\label{eq_MHE_problem}
    \end{subequations}
where the minimization of estimation errors is performed for a $T$-long trajectory in closed-loop rather than at a single snapshot. A merit of the SLS framework lies in the transformation of the non-convex problem of designing the observer gain $L$ into a convex program over the closed-loop response $(\Phi_x,\Phi_y)$ \cite{anderson2019system}. Under the Gaussian assumption, i.e., $\mathbb{P}^{\star}=\mathcal{N}(0,\Sigma)$, the objective \eqref{MHE_MSE} becomes:
\begin{equation}
    \mathbb{E}_{ \xi\sim  \mathbb{P}^{\star}}\left[ \lVert \bar e \rVert^2 \right]= \left\lVert  \begin{bmatrix}
       \Phi_x&\Phi_y
   \end{bmatrix} \begin{bmatrix}
        \bar B\\ -\bar D
   \end{bmatrix} \Sigma^{1/2} \right\rVert_F^2. \label{eq_H2_MSE}
\end{equation}
The $\mathcal{H}_2$ observer gain is given by $L^{\star}={(\Phi_x^{\star})}^{-1}\Phi^{\star}_y$, which is attainable from the optimal solution:
\begin{equation}
    \begin{aligned}
  \left(\Phi_x^{\star},\Phi^{\star}_y  \right)\in \arg \min_{\Phi_x,\Phi_y}
    &~\eqref{eq_H2_MSE} 
    \\
   {\rm  s.t.}&~  \eqref{MHE_structure}, \eqref{MHE_lower_triangle}
\end{aligned} \label{eq_H2}
\end{equation}
In particular, if the initial estimation error $e(t_0)$ and disturbance process $\{w(i)\}_{i=t_0}^{t_0+T-1}$ are mutually independent, rendering a block-diagonal $\Sigma$, the observer gain solved from \eqref{eq_H2} turns out to equal the well-known Kalman gain \cite{brouillon2023minimal}.

\section{Main Results} \label{mainr} 
\subsection{Sinkhorn DRSE Design Problem}
While the Gaussian assumption offers computational tractability, its validity is challenged in engineering practice. When the real-world disturbances exhibit non-Gaussianity, such as heavy tails and multi-modality, the accuracy and reliability of $\mathcal{H}_2$ design \eqref{eq_H2} can be compromised considerably \cite{stojanovic2016robust}. To cope with uncertainty in the disturbance distribution, we adopt the DRO paradigm to design a state estimator. Its core is to construct a so-called ambiguity set $\mathcal{D}$ to capture all admissible distributions and optimize the worst-case objective. Among various options of ambiguity sets, the Wasserstein ambiguity set has been extensively utilized due to its computational tractability and clear interpretability \cite{mohajerin2018data,gao2024wasserstein}. The Wasserstein 
distance underlying this ambiguity set is induced by an optimal transport 
problem, which is introduced below.
\begin{definition}[Optimal transport distance, \cite{kantorovich1958space}]
Consider a support set $\Xi$ and distributions $\mathbb{P}\in \mathcal{P}{(\Xi)}$, $\mathbb{P}'\in \mathcal{P}{(\Xi)}$, then the optimal transport distance between $\mathbb{P}$ and $\mathbb{P}'$ is defined as:
\begin{equation*}
\mathcal{W}(\mathbb{P},{\mathbb{P}}') =  \inf_{\pi\in \Pi(\mathbb{P},\mathbb{P}')}\mathbb{E}_{(\xi,\xi')\sim\pi} \left[ c(\xi,\xi')  \right],
\end{equation*}
where $\Pi \subseteq \mathcal{P}{(\Xi\times\Xi)}  $ is the transport plan composed of all joint distributions of $\xi$ and $\xi'$ with marginal distributions $\mathbb{P}$ and $\mathbb{P}'$, respectively. The transport cost $c(\xi,\xi')\geq 0$ is a lower semi-continuous function used to measure the distance between $\xi$ and $\xi'$, satisfying $c(\xi,\xi')= 0$ if and only if $\xi=\xi'$. If $c(\xi,\xi')= \lVert\xi-\xi'  \rVert^p$ with $p\geq1$, then $\mathcal{W}(\mathbb{P},{\mathbb{P}}')^{1/p}$ coincides with the $p$-Wasserstein distance. In this work, we use the squared $2$-norm $c(\xi,\xi') = \lVert \xi-\xi'\rVert^2$ as the transport cost such that $\mathcal{W}(\mathbb{P},\mathbb{P}')$ represents the squared
$2$-Wasserstein distance.
\end{definition}

The Wasserstein distance quantifies the difference between two distributions via optimal transport. On this basis, the Wasserstein ambiguity set can be defined as a ball in $\mathcal{P}(\Xi)$ centered at a reference distribution. 

\begin{definition}[Wasserstein ambiguity set, \cite{mohajerin2018data}] \label{def_amb_W}
Given $N$ i.i.d. samples of uncertainty $\{ \hat\xi_i\}_{i=1}^N$, the Wasserstein ambiguity set $\mathcal{D}(\hat{\mathbb{P}}_N,\theta)$ is defined as:
\begin{equation}
    \mathcal{D}(\hat{\mathbb{P}}_N,\theta) = \left\{\mathbb{P}\in \mathcal{P}(\Xi)\left|\mathcal{W}(\hat{\mathbb{P}}_{N},\mathbb{P})\leq\theta\right.\right\},
\end{equation}
where $\hat{\mathbb{P}}_{N} = \frac{1}{N} \sum_{i=1}^{N} \delta_{\hat{\xi}_i}$  is the empirical distribution and $\theta$ is the radius of the Wasserstein ball.
\end{definition}

The Wasserstein DRO with $\mathcal{D}(\hat{\mathbb{P}}_N,\theta)$ has proven to be asymptotically equivalent to variation regularization \cite{gao2024wasserstein}, which has played a central role in plenty of machine learning and statistical regression techniques \cite{hoerl1970ridge, park2008bayesian, tian2022comprehensive}. However, it is shown in \cite{gao2023distributionally,yue2022linear} that the worst-case distribution is supported on at most $N+1$ discrete points. This often deviates from the underlying continuity of probability distributions in reality and gives rise to a conservative decision. To alleviate over-conservatism, the Sinkhorn distance, as an entropy-regularized variant of the Wasserstein distance, has been utilized to construct ambiguity sets in DRO.
\begin{definition}[Sinkhorn distance, \cite{wang2025sinkhorn}] \label{def_sink}
Consider a measurable support set $\Xi$ and distributions $\mathbb{P}\in \mathcal{P}{(\Xi)}$, $\mathbb{P}'\in \mathcal{P}{(\Xi)}$, and let $\alpha\in \mathcal{M}(\Xi),~\beta \in \mathcal{M}(\Xi)$ be two reference measures such that $\mathbb{P}\ll \alpha$, $\mathbb{P'}\ll \beta$. Given the entropic regularization parameter $\epsilon \geq 0$, the Sinkhorn distance between distributions $\mathbb{P}$ and $\mathbb{P}'$ is defined as:
\begin{equation}
\begin{split}
    &~ \mathcal{W}_{\epsilon}(\mathbb{P},{\mathbb{P}}') \\
    = & \inf_{\pi\in \Pi(\mathbb{P},\mathbb{P}')}\left\{\mathbb{E}_{(\xi,\xi')\sim\pi} \left[ c(\xi,\xi')  \right]+\epsilon \cdot {\rm KL}\left( \pi| \alpha \odot\beta \right)
\right\},
\end{split} \label{eq_def_sink}
\end{equation}
where ${\rm KL}\left( \pi| \alpha \odot\beta \right)$ is the KL-divergence of $\pi$ with respect to the product measure $ \alpha \odot\beta $:
\begin{equation}
      {\rm KL}\left( \pi | \alpha \odot\beta \right)= \mathbb{E}_{(\xi,\xi')\sim\pi} \left[ \log \frac{{\rm d}\pi(\xi,\xi') }{{\rm d} \alpha(\xi){\rm d} \beta(\xi')} \right].
\end{equation}
Here, $  \frac{{\rm d}\pi(\xi,\xi') }{{\rm d} \alpha(\xi){\rm d} \beta(\xi')}$ stands for the density ratio of $\pi$ with respect to $ \alpha \odot\beta $ evaluated at $(\xi,\xi')$. 
\end{definition}

As a generalization of the Wasserstein ball $\mathcal{D}(\hat{\mathbb{P}}_N,\theta)$, the Sinkhorn ambiguity set is defined as follows.  

\begin{definition}[Sinkhorn ambiguity set, \cite{wang2025sinkhorn}] \label{def_amb_S}
Given the empirical distribution $\hat{\mathbb{P}}_N$ defined in Definition \ref{def_amb_W} and two reference measures $\alpha\in \mathcal{M}(\Xi)$, $\beta\in \mathcal{M}(\Xi)$ satisfying $\hat{\mathbb{P}}_N\ll \alpha$, $\mathbb{P}\ll \beta$, the Sinkhorn ambiguity set is defined as:
\begin{equation}
    \mathcal{D}_{\epsilon}(\hat{\mathbb{P}}_N,\theta) = \left\{\mathbb{P}\in \mathcal{P}(\Xi)\left|\mathcal{W}_{\epsilon}(\hat{\mathbb{P}}_{N},\mathbb{P})\leq\theta\right.\right\}. \label{eq_ambi_sinkhorn}
\end{equation}
\end{definition}

Intuitively, the entropic regularizer in $\mathcal{D}_{\epsilon}(\hat{\mathbb{P}}_N,\theta)$ imposes a penalty on the transport plan $\pi$ different from  $ \alpha \odot\beta $ to shrink $\mathcal{D}(\hat{\mathbb{P}}_N,\theta)$. Since every candidate distribution $\mathbb{P}\ll\beta$ is absolutely continuous with respect to $\beta$, using a continuous $\beta$ can effectively avoid a sparse $\pi$ by ruling out $\mathbb{P}$ that assigns positive probability to individual points, thereby focusing on smooth and continuous distributions that are more realistic.
It is suggested in \cite{wang2025sinkhorn} that one can choose $\beta$ to be the Gaussian measure or Lebesgue measure if the true distribution $\mathbb{P}^{\star}$ is known to be continuous. As for the selection of $\alpha$, it has been proved that arbitrary choice of $\hat{\mathbb{P}}_N\ll \alpha$ is equivalent up to a constant, i.e., ${\rm KL}\left( \pi| \alpha \odot\beta \right)={\rm KL}\left( \pi \big| \hat{\mathbb{P}}_N \odot\beta \right)+
{\rm KL}\left(\hat{\mathbb{P}}_N| \alpha  \right)$, because the empirical distribution $\hat{\mathbb{P}}_N$ is fixed \cite{wang2025sinkhorn}. In this work, $\xi$ is believed to be continuously supported on $\mathbb{R}^{n_\xi}$, and thus we choose $\alpha=\hat{\mathbb{P}}_N$ and $\beta$ to be a Gaussian distribution $\mathcal{N}(0,\Sigma)$:
\begin{equation}
  {\rm d}\beta(\xi)=\frac{\exp{\left\{-{\xi^\top \Sigma^{-1}\xi}/{2}\right\}} }{\sqrt{(2\pi)^{n_\xi} {\rm det}(\Sigma)}}  {\rm d}\nu(\xi),   \label{eq_Gaussian}
\end{equation}
where $\nu$ is the Lebesgue measure and we assume $\mu=0$ without loss of generality. Using the ambiguity set \eqref{eq_ambi_sinkhorn}, the Sinkhorn DRSE design problem is then formulated as the following min-max problem:
\begin{subequations} 
\begin{align}
    \min_{\Phi} \sup_{\mathbb{P}\in\mathcal{D}_{\epsilon}(\hat {\mathbb{P}}_N,\theta)} &~\mathbb{E}_{ \xi\sim  \mathbb{P}}[ \xi^\top P(\Phi)\xi ]\\
   {\rm  s.t.}~~~
   &~  \Phi R=I,~\Phi= \begin{bmatrix}
        \Phi_x & \Phi_y
    \end{bmatrix} \label{eq_primal_SLS1}\\
   &~  \Phi_x \text{ and } \Phi_y \text{ are block lower triangular} \label{eq_primal_SLS2}
\end{align} \label{eq_primal}
\end{subequations}
where $P(\Phi)=Q^\top\Phi^\top \Phi Q$, $R= \begin{bmatrix}
      (I-Z\bar A)^\top   &     (\bar CZ )^\top
\end{bmatrix}^\top$ and $ Q=\begin{bmatrix}
        \bar B^\top& -\bar D^\top
   \end{bmatrix}^\top$ are defined for succinctness. 

\begin{remark}
The worst-case objective in \eqref{eq_primal} bounds the expected estimation error under every distribution in $\mathcal{D}_{\epsilon} (\hat{\mathbb{P}}_N,\theta)$, thereby providing
distributional robustness for the resulting estimator. In particular, this bound remains valid under distributional shifts, including time-varying changes that cause the actual distribution to differ from that represented by the historical samples $\{\hat{\xi}_i\}_{i=1}^N$, as long as the shifted distribution remains within $\mathcal{D}_{\epsilon}(\hat{\mathbb{P}}_N,\theta)$.
\end{remark}

\subsection{Finite-Sample Probabilistic Guarantee}
In this subsection, we establish a novel performance guarantee under two mild assumptions on the unknown true distribution $\mathbb{P}^{\star}$.

\begin{assumption}[\cite{mohajerin2018data}]\label{assu_light}
The distribution $\mathbb{P}^{\star}$ is a light-tailed distribution, that is, there exists a constant \chadded
{$0<a\neq2$} satisfying
\begin{equation*}
\psi(a)= \mathbb{E}_{\xi \sim \mathbb{P}^{\star}}\left[\exp\{ \lVert\xi\rVert^a\}  \right]< \infty. 
\end{equation*}
\end{assumption}
\begin{assumption}\label{assu_nodirac}
The true probability measure $\mathbb{P}^{\star}$ is absolutely continuous
with respect to the Lebesgue measure $\nu$ and admits a bounded density, i.e., there exists a constant $b > 0$ such that
\begin{equation*}
\operatorname*{ess}\sup_{\xi\in\Xi} \frac{{\rm d}\mathbb{P}^{\star}(\xi) }{{\rm d}
\nu(\xi)}\leq b. 
    \end{equation*}
\end{assumption}

Assumption \ref{assu_light} is commonly adopted in Wasserstein DRO to prevent the sample average loss not converge to its true expectation \cite{shafieezadeh2018wasserstein,brownlees2015empirical}. In addition, an infinite spike $\delta_{\tilde{\xi}}$ on any $\tilde{\xi}$ for $\mathbb{P}^{\star}$ is physically unrealistic and can be eliminated by Assumption \ref{assu_nodirac}. In this way, $\mathbb{P}^{\star}$ is supposed to be absolutely continuous without any discrete components. We are now ready to establish the finite-sample probabilistic guarantee of the Sinkhorn ambiguity set $\mathcal{D}_{\epsilon}(\hat{\mathbb{P}}_N,\theta)$. In order not to disturb the flow of exposition, all proofs in this section are deferred to Appendix.

\begin{theorem} \label{theo_fin}
Assume $n_\xi \neq 4$ and the true distribution $\mathbb{P}^\star$ satisfies Assumptions \ref{assu_light} and \ref{assu_nodirac}. Then its mean and covariance are finite and denoted by $\mathbb{E}_{\xi \sim \mathbb{P}^{\star}}[\xi ] = \mu^{\star}$ and $\mathbb{E}_{\xi \sim \mathbb{P}^{\star}}[(\xi-\mu^{\star})(\xi-\mu^{\star})^\top] = \Sigma^{\star}$. For a prescribed confidence level $\eta \in (0,1)$ and sample size $N$, the radius $\theta_N(\eta)$ is specified as follows. 
If $0<a<2$, choose any $\eta'\in(0,\eta)$ and $\bar a\in(0,a)$, and set 
\begin{equation}
\begin{split}
      &~~~  \theta_N(\eta)\\
&=\begin{cases}
 \max\Bigg\{ 
\frac{1}{N}\left( \frac{\log\left( \frac {c_1} {\left(\eta-\eta'\right)}\right)}{c_2}\right)^{\frac{2}{a-\bar a}},\\
\left( \frac{\log\left(\frac{ c_1}{\eta'}\right)}{c_2N}\right)^{\frac{1}{\max\{n_\xi/2,2\}}}\Bigg\} +\epsilon m,
& \text{if } N \geq  \left(\frac{ \log\left(c_1 /\eta\right)}{c_2}\right)^{\frac{2}{a}}\\
\frac{1}{N} \left(\frac{ \log\left(c_1 /\eta\right)}{c_2}\right)^{2/a}+\epsilon m,  & \text{if } N <  \left(\frac{ \log\left(c_1 /\eta\right)}{c_2}\right)^{\frac{2}{a}} \\
\end{cases}
\end{split} \label{eq_theo_fin_thetaN1}  
\end{equation}
where $c_1>0$ and $c_2>0$ are constants depending only on $(\psi(a),a,\bar a,n_\xi)$, and
\begin{equation}
 \begin{split}
   m&=  \log (bN) +\frac{n_\xi}{2}\log(2\pi) +\frac{1}{2}\log\det\left(\Sigma \right)   \\
   & ~~~~+\frac{1}{2} (\mu^{\star})^\top \Sigma^{-1} {\mu^{\star}}+ \frac{1}{2}{\rm Tr}\left\{\Sigma^{-1} \Sigma^{\star} \right\}.\end{split} \label{eq_theo_fin_m}
\end{equation}
If $a>2$, set
\begin{equation}
\begin{split}
      &~~~  \theta_N(\eta)\\
&=\begin{cases}
\left( \frac{\log\left( c_1/ \eta\right)}{c_2N}\right)^{\frac{1}{\max\{n_\xi/2,2\}}}+\epsilon m, & \text{if } N \geq  \frac{\log\left( c_1 /\eta \right)}{c_2}\\
\left( \frac{\log\left( c_1 /\eta \right)} {c_2N}\right)^{2/a} +\epsilon m,  & \text{if } N <  \frac{\log\left( c_1 /\eta\right)}{c_2} \\
\end{cases}~
\end{split} \label{eq_theo_fin_thetaN2}  
\end{equation}
Then, the resulting Sinkhorn ambiguity set satisfies
\begin{equation}
    \mathbb{P}^{N}\left\{\mathbb{P} ^{\star} \in  \mathcal{D}_{\epsilon}\left(\hat {\mathbb{P}}_N,\theta_N(\eta)
    \right)\right\} \geq 1-\eta. \label{eq_theo_fin_bound}
\end{equation}
\end{theorem}
 
\begin{remark}
Theorem \ref{theo_fin} generalizes the classical finite-sample probabilistic guarantee of the Wasserstein ambiguity set in \cite[Theorem 3.4]{mohajerin2018data}, which can be seen as a special case of Theorem \ref{theo_fin} by setting $\epsilon=0$.
\end{remark}

As an immediate consequence of Theorem \ref{theo_fin}, the performance guarantee of solving the Sinkhorn DRSE design problem \eqref{eq_primal} can be attained.

\begin{corollary} \label{coro}
Given the optimal design $\Phi^{\star}=\begin{bmatrix}
    \Phi_x^{\star}&\Phi_y^{\star}
\end{bmatrix}$ and the optimal value $J^{\star}$ of the Sinkhorn DRSE problem \eqref{eq_primal} with radius $\theta=\theta_N(\eta)$, the following bound on the estimation performance holds under the underlying true distribution $\mathbb{P}^{\star}$: 
\begin{equation*}
\mathbb{P}^N\left\{ \mathbb{E}_{\xi\sim\mathbb{P}^{\star}} \left[\left\lVert \bar e \left({(\Phi_x^{\star})}^{-1}\Phi^{\star}_y,\xi \right) \right\rVert^2\right] \leq J^{\star}\right\}
\geq  1-\eta.
\end{equation*}
\end{corollary}

The theoretical bounds $ \theta_N(\eta)$ established in \eqref{eq_theo_fin_thetaN1} and \eqref{eq_theo_fin_thetaN2} are overly conservative and are not practically computable because the parameters $(a,c_1,c_2,\mu^\star,\Sigma^\star)$ are unknown. Therefore, it is advisable to use data-driven methods, e.g., cross-validation, to fine-tune the radius $\theta$ rather than directly using $ \theta_N(\eta)$.

\subsection{Tractable Reformulation of Sinkhorn DRSE Problem}
The Sinkhorn DRSE design problem \eqref{eq_primal} is formulated as a min-max optimization involving a worst-case expectation. For tractability, a finite-dimensional convex programming reformulation is derived.

\begin{theorem} \label{theo_dual}
The primal problem \eqref{eq_primal} is feasible if and only if 
\begin{equation}
\begin{aligned}
     \theta \geq   \theta_0 &  \triangleq \frac{\epsilon}{2}\log \det \Sigma - \frac{\epsilon n_\xi}{2}\log\left(\frac{\epsilon}{2}\right)\\
  & ~~~~~ +\frac{\epsilon}{2}\log \det \Omega + \frac{1}{N}\sum_{i=1}^N \hat{\xi}_i^\top \left(I-\Omega^{-1}\right)\hat{\xi}_i 
    , \label{eq_theo_dual_fea}
\end{aligned}
\end{equation}
and admits the following exact reformulation:
 \begin{subequations}
\begin{align}
\begin{split}
      \inf_{\Phi, P, q_i,\lambda\geq0} &~ \lambda \theta-\frac{\lambda\epsilon}{2}\log \det \Sigma + \frac{\lambda\epsilon n_\xi}{2}\log\left(\frac{\lambda\epsilon}{2}\right)\\
    &~~~~~~~-\frac{\lambda\epsilon}{2}\log \det\left(\lambda \Omega -P\right)+\frac{1}{N}\sum^N_{i=1}q_i 
\end{split}   \label{eq_main_dual_O}\\
    {\rm s.t.}~~~ &~ \begin{bmatrix}
       \lambda\Omega - P  & \lambda\hat{\xi}_i \\ 
       \lambda\hat{\xi}_i^\top & \lambda \hat{\xi}_i^\top \hat{\xi}_i + q_i
    \end{bmatrix} \succeq 0,~ i\in \mathbb{N}_{1:N}\label{eq_main_dual_C1} \\
    &~ \begin{bmatrix}
        P & Q^\top \Phi^\top \\
       \Phi Q & I
    \end{bmatrix} \succeq 0 \label{eq_main_dual_C2} \\
       &~ \lambda\Omega - P \succeq \kappa I \label{eq_main_dual_C3}\\
      &~  \Phi   R=I, \Phi= \begin{bmatrix}
        \Phi_x & \Phi_y
    \end{bmatrix}\label{eq_main_dual_C4} \\ 
    &~ \Phi_x \text{ and } \Phi_y \text{ are block lower triangular}\label{eq_main_dual_C5}
\end{align} \label{eq_main_dual}
\end{subequations}
where $\Omega = I +\frac{ \epsilon}{2}\Sigma^{-1}\succ0$ and $\kappa>0$ is a sufficiently small positive number.
\end{theorem}

The dual reformulation \eqref{eq_main_dual} is a finite-dimensional convex program with LMI constraints, where the objective function is provably convex albeit nonlinear \cite{cescon2025data}. To shed light on the reduced conservatism of Sinkhorn DRSE, its worst-case distribution is characterized as follows. 

\begin{theorem} \label{theo_worst}
Given the optimal solution $(\lambda^{\star},\Phi^{\star})$ of \eqref{eq_main_dual}, the worst-case distribution $\tilde{\mathbb{P}}$ of the Sinkhorn DRSE \eqref{eq_primal} is 
\begin{equation}
\begin{aligned}
    {\rm d}\tilde{\mathbb{P}}(\xi')= \frac{1}{N} \sum_{i=1}^N v_i\cdot \exp\left\{  (\xi')^{\top} U \xi' + u_i^{\top}  \xi'\right\}{\rm d} \nu(\xi'), 
\label{eq_theo_worst}
\end{aligned}
\end{equation}
where 
\begin{equation}
\begin{aligned}
    U =&  \frac{ 1}{\lambda^\star \epsilon}\left[P(\Phi^{\star})-\lambda^{\star}\Omega\right],~ 
    u_i =   \frac{ 2}{ \epsilon} \hat{\xi}_i,\\
    v_i = &  \frac{\sqrt{\det\left[\lambda^\star \Omega-P(\Phi^{\star})\right]}}{{(\pi\lambda^\star\epsilon)}^{n_\xi/2}} \cdot \exp\left\{\frac{1}{\epsilon^2}\hat{\xi}_i^{\top} U^{-1} \hat{\xi}_i\right\}. \label{theo_worst_para}
\end{aligned}
\end{equation}
\end{theorem}

Theorem \ref{theo_worst} highlights that the worst-case distribution $\tilde{\mathbb{P}}$ of the Sinkhorn DRSE is continuously supported on $\mathbb{R}^{n_\xi}$, thereby better characterizing realistic continuous distributions. 

Before closing this subsection, we exploit the reformulated problem \eqref{eq_main_dual} to further characterize the limiting properties of the proposed Sinkhorn DRSE. We first consider the simple case of $\epsilon = 0$, where the entropic regularization in the Sinkhorn distance vanishes, and the Sinkhorn ambiguity set $\mathcal{D}_{\epsilon}(\hat{\mathbb{P}}_N,\theta)$ boils down to the Wasserstein ambiguity set $\mathcal{D}(\hat{\mathbb{P}}_N,\theta)$. Besides, it follows from the L'Hospital's rule that
$\lim_{\epsilon\to 0^+} \lambda\epsilon\log(\lambda\epsilon/2)=0$. Henceforth, the nonlinear terms induced by entropic regularization in \eqref{eq_main_dual_O} vanish, and \eqref{eq_main_dual} reduces to the Wasserstein DRSE with finite $\theta>0$.
When $\epsilon=\theta=0$, \eqref{eq_primal} further boils down to a sample-average approximation problem, whose optimal solution is known to coincide with the $\mathcal{H}_2$ design \eqref{eq_H2} \cite{mandel2011convergence,lorentzen2011iterative}. 

Next, we consider the limiting regime $\epsilon\rightarrow\infty$, which can be distinguished into the following cases.
\begin{itemize}
    \item \textbf{Case 1:} $\theta$ is finite and $\epsilon \rightarrow{\infty}$. As already studied in \cite{cescon2025data}, if $ \theta\geq{\rm Tr}\left\{\Sigma\right\}+\frac{1}{N}\sum^N_{i=1}\hat{\xi}_i^\top \hat{\xi}_i$, \eqref{eq_primal} also reduces to the $\mathcal{H}_2$ design \eqref{eq_H2} because $\mathcal{D}_{\epsilon}(\hat{\mathbb{P}}_N,\theta)=\{\beta\}$ is a singleton. Otherwise, the set $\mathcal{D}_{\epsilon}(\hat{\mathbb{P}}_N,\theta)$ is empty, rendering \eqref{eq_primal} infeasible.
    \item \textbf{Case 2:} $\theta  \rightarrow{\infty}$, $\epsilon \rightarrow{\infty}$ but $\theta/\epsilon \rightarrow \theta_{\rm KL} \geq0$ is finite. This case was not considered by \cite{cescon2025data}. Similar to Case 1, the entropic regularizer $\epsilon{\rm KL}(\pi|\alpha\odot \beta)$ outweighs the transport cost $\mathbb{E}_{(\xi,\xi')\sim\pi} \left[ c(\xi,\xi')  \right]$, and thus the Sinkhorn ambiguity set becomes $\mathcal{D}_{\epsilon}(\hat{\mathbb{P}}_N,\theta)=\left\{ {\rm KL}(\mathbb{P}| {\beta}) \leq \theta/\epsilon=\theta_{\rm KL}\right\}$, which is a KL ball centered at the Gaussian distribution $\beta$. The Sinkhorn DRSE design problem \eqref{eq_primal} thus reduces to a KL DRSE design problem, where the $\mathcal{H}_2$ design is known to be optimal \cite{levy2012robust}. 
    \end{itemize}
    
In a nutshell,  a smaller $\epsilon$ leads to a more robust design against distributional uncertainty, albeit at the cost of being more conservative. Conversely, using a larger $\epsilon$ imposes a heavier entropic penalty on the transport plan, and consequently, the resulting Sinkhorn estimator bears a closer resemblance to the $\mathcal{H}_2$ estimation. Thus, the proposed approach strikes a flexible balance between the conventional $\mathcal{H}_2$ design under the Gaussian assumption and the data-centric Wasserstein DRSE with guaranteed distributional robustness.

\subsection{First-Order Solution Algorithm}
Since the objective $J(P,\lambda,q)$ contains the nonlinear log-determinant term $\lambda\log\det(\lambda\Omega-P)$, problem \eqref{eq_main_dual} is convex but is not a standard semi-definite programming (SDP).
To effectively solve this problem, we propose a first-order solution algorithm in this subsection. For notational brevity, we first write the objective \eqref{eq_main_dual_O} as a function $J\left(P, \lambda, q\right)$ and regroup its terms as follows: 
\begin{equation}
\begin{aligned}
      & J(P,\lambda,q)  =  \underbrace{\lambda \theta-\frac{\lambda\epsilon}{2}\log \det(\Sigma)}_{\triangleq f(\lambda)}
+\underbrace{\frac{1}{N} \sum^N_{i=1}q_i}_{=1_N^\top q/N\triangleq g \left(q\right),~q=[q_1~\cdots ~q_N]^\top}\\
    &+  \underbrace{ \frac{\lambda\epsilon n_\xi}{2}\log\left(\frac{\lambda\epsilon}{2}\right)   -\frac{\lambda\epsilon}{2}\log \det\left(\lambda \Omega-P\right)}_{\triangleq h(P,\lambda)}.
\end{aligned}
\end{equation}
Specifically, for any scaling factor $t > 0$, we have:
\begin{equation}
    \begin{aligned}
{J}\left(tP,t\lambda,tq \right)&= f(t\lambda)+g\left(tq\right)+h(tP,t\lambda)\\
    &=t f(\lambda)+tg\left(q\right)+ (t\lambda)  \bar h \left(\frac{tP}{t\lambda}\right)\\
    &= t\left[f(\lambda)+g\left(q\right)+\lambda  \bar h \left(\frac{P}{\lambda}\right)\right]\\
    &=t\cdot {J}\left(P,\lambda,q \right),\\
\end{aligned} \label{eq_homogeneity}
\end{equation}
where the second equality follows from the fact that $f(\lambda)$ and $g(q)$ are linear, and $h(P,\lambda)= \lambda \bar h(P/\lambda)$ is a perspective of the function $\bar h(P)$ \cite{cescon2025data,boyd2004convex}, defined as:
\begin{equation*}
    \bar h(P) =    \frac{\epsilon n_\xi}{2}\log \left(
    \frac{\epsilon}{2} \right)-\frac{\epsilon}{2} \log  \det(\Omega-P).
\end{equation*}
The relation \eqref{eq_homogeneity} characterizes the \textit{$1$-degree homogeneity} of $ {J}\left(P,\lambda,q \right)$, which enables efficient minimization of the objective of \eqref{eq_main_dual} via iterative linear approximations and makes the Frank-Wolfe algorithm a suitable option \cite{frank1956algorithm,jaggi2013revisiting}. However, the Lagrange multiplier $\lambda$ and the epigraphical variables $(P,q)$ in \eqref{eq_main_dual} are essentially unbounded, while the Frank-Wolfe procedure theoretically entails a compact feasible region to prevent the linearized subproblem from being unbounded. To overcome this obstacle, we establish upper bounds for the optimal solution $\lambda^{\star}$ and $(P^\star,q^\star)$.

\begin{lemma} \label{lemma_upper}
Assume the dual problem \eqref{eq_main_dual} is strictly feasible, i.e., $\theta > \theta_0$, and admits a known feasible solution $(\Phi^{(0)}, P^{(0)}, \lambda^{(0)}, q^{(0)})$. The optimal Lagrange multiplier $\lambda^{\star}$ is upper bounded by:
\begin{equation}
        \lambda^{\star} \leq \frac{\hat J(P^{(0)}, \lambda^{(0)})}{\theta - \theta_0}, 
        \label{eq_lemma_upper_lambda_star}
\end{equation}
where the constant $\theta_0$ is defined in \eqref{eq_theo_dual_fea} and the function ${\hat J(P, \lambda)} $ is defined as:
 \begin{align*}
      & {\hat J(P, \lambda)} =f(\lambda)+h( P,\lambda)+\hat g(P,\lambda),\\
      & \hat g(P,\lambda) =\frac{1}{N} \sum_{i=1}^N \left[ \lambda^2 \hat{\xi}_i^\top (\lambda \Omega - P)^{-1} \hat{\xi}_i - \lambda \hat{\xi}_i^\top \hat{\xi}_i \right]. 
       \end{align*}
\end{lemma}

\begin{lemma}
    \label{lemma_upper_else}
Assume the problem \eqref{eq_main_dual} is strictly feasible. Then, its optimal solution $(P^\star,\lambda^\star,q^\star) $ satisfies the bound $ \underline{ P} \preceq P^{\star}\preceq \bar P$, $ \underline{ \lambda} \leq \lambda^{\star}  \leq \bar \lambda$ and $\underline{q}\leq q^{\star} \leq \bar q$, where the upper bound $\lambda^{\star}\leq  \bar \lambda$ is given in Lemma \ref{lemma_upper} and
\begin{align}
&\underline{P}= 0,~
\bar P=\bar  \lambda \Omega -\kappa I, ~
\underline{\lambda}=\bar\sigma \left(\kappa \Omega^{-1} \right) , \label{lemma_upper_else_bound_1}\\
&\underline{q}=\begin{bmatrix}
    \underline{q_1}&\cdots&\underline{q_N}
\end{bmatrix}^\top,~\underline{q_i}=\underline{\lambda}^2 \hat{\xi}_i^\top (\bar \lambda \Omega - \underline{P})^{-1} \hat{\xi}_i -\bar {\lambda }\hat{\xi}_i^\top \hat{\xi}_i ,\\ 
&\bar{q}=\begin{bmatrix}
    \bar{q}_1&\cdots&\bar{q}_N
\end{bmatrix}^\top,~\bar q_i=\bar\lambda^2 \hat{\xi}_i^\top  \hat{\xi}_i/\kappa - \underline{ \lambda} \hat{\xi}_i^\top \hat{\xi}_i .\label{lemma_upper_else_bound_4}
\end{align}
\end{lemma}

Building upon Lemmas \ref{lemma_upper} and \ref{lemma_upper_else}, one can impose the constraints into the feasible set of \eqref{eq_main_dual} and define a subset $\mathcal{S}$ of the feasible region, whose projection ${\rm proj}_{P,\lambda,q}(\mathcal{S})$ is compact: 
\begin{equation}
    \mathcal{S}= \left\{ \left(\Phi,P,\lambda,q\right) \left|\begin{aligned}
        &\underline{P} \preceq P \preceq \bar{P},~\underline{\lambda} \leq\lambda \leq \bar{\lambda},\\
        &\underline{q} \leq  q\leq \bar{q},~\eqref{eq_main_dual_C1}-\eqref{eq_main_dual_C5}.
    \end{aligned}
    \right. \right\}. \label{eq_feasible_S}
\end{equation}
Since the added constraints are valid for the optimal solution, $(\Phi^{\star},P^{\star},\lambda^{\star},q^{\star}) \in \mathcal{S}$ still holds and preserves the optimality of \eqref{eq_main_dual}. Now we are ready to apply the Frank-Wolfe algorithm to iteratively solve \eqref{eq_main_dual}. Given the current point $(P^{(k)},\lambda^{(k)},q ^{(k)})$ at the $k$th iteration, we define $J^{(k)} \triangleq J(P^{(k)},\lambda^{(k)},q ^{(k)})$ and solve the following direction-finding subproblem: 
\begin{equation}
\begin{aligned}
  \displaystyle    \min_{\Phi, \tilde{P},\tilde{\lambda}, \tilde q} &~\left\langle \tilde P,  \nabla_P J^{(k)}\right \rangle+\tilde{\lambda }\nabla_{\lambda} {J} ^{(k)}+\tilde{q}^\top \nabla_{q}  {J} ^{(k)}\\
   {\rm s.t.}~&~{\left(\Phi, \tilde{P},\tilde{\lambda}, \tilde q\right) \in \mathcal{S}}
\end{aligned} \label{eq_find_direction}
\end{equation}
where $\nabla_{P}  {J}^{(k)}= {\lambda^{(k)}\epsilon}( \lambda^{(k)}\Omega - P^{(k)})^{-1}/{2} $, $\nabla_{q}  {J}^{(k)}= {1}_N/N$ and 
 \begin{align*}
 \nabla_{\lambda}  {J}^{(k)}&=
 \theta - \frac{\epsilon}{2}\log\det(\Sigma) + \frac{\epsilon n_\xi}{2}\left[\log\left(\frac{\lambda^{(k)}\epsilon}{2}\right) + 1\right]\\
   & ~~~- \frac{\epsilon}{2}\log\det\left(\lambda^{(k)}\Omega-P^{(k)}\right) \\
   & ~~~- \frac{\lambda^{(k)}\epsilon}{2}{\rm Tr}\left\{\left(\lambda^{(k)}I-P^{(k)}\Omega^{-1}\right)^{-1}\right\}, 
\end{align*}
Indeed, the linear minimization oracle \eqref{eq_find_direction} is an SDP. It can be effectively solved by the interior-point method \cite{toh2018some}, and the computation cost is approximately $\mathcal{O}(Nn_\xi(n_\xi^2+N)^{3})$ arithmetic operations with $n_\xi = Tn_w+n_x$ \cite{pascal1995lmi}. After solving \eqref{eq_find_direction}, its optimal solution $(\tilde{P}^{\star},\tilde{\lambda}^{\star},\tilde{q}^{\star} )$ indicates the descent direction. Using step size $\rho^{(k)}=2/(k+2)$, the updating rule of $({P},{\lambda},q)$ can be expressed as: 
\begin{align*}
  &   P^{(k+1)}= (1-\rho^{(k)}) P^{(k)}+\rho^{(k)}\tilde{P}^{\star},\\
  &  \lambda^{(k+1)}= (1-\rho^{(k)})\lambda^{(k)}+\rho^{(k)}\tilde{\lambda}^{\star},\\
  & q^{(k+1)} = (1-\rho^{(k)})q^{(k)}+\rho^{(k)}\tilde{q}^{\star}.
\end{align*}
In the standard Frank-Wolfe algorithm, $(P^{(k)},q^{(k)},\lambda^{(k)})$ should be updated at every iteration. However, in our implementation, only $(P^{(k)},\lambda^{(k)})\in {\rm proj}_{P,\lambda}(\mathcal{S})$ need to be updated explicitly. This is because the gradient $\nabla_q J^{(k)}=\frac{1}{N}{1}_N$ does not depend on the current value of $q^{(k)}$, and thus there is no need to update $q^{(k)}$. Due to the convexity of $\mathcal{S}$, this updating rule preserves the strict feasibility of $(P^{(k+1)}, \lambda^{(k+1)})$ via convex combination of $(P^{(k)},\lambda^{(k)})\in {\rm proj}_{P,\lambda}(\mathcal{S})$ and $(\tilde{P}^{\star},\tilde{\lambda}^{\star})\in {\rm proj}_{P,\lambda}(\mathcal{S})$, thereby obviating the need for cumbersome projections \cite{jaggi2013revisiting}. In the $k$th iteration, we denote the optimal value of \eqref{eq_find_direction} as $\tilde J^{(k)}\triangleq \langle \tilde P^{\star},  \nabla_P J^{(k)} \rangle+\tilde{\lambda }^{\star}\nabla_{\lambda} {J} ^{(k)}+(\tilde{q}^{\star})^\top \nabla_{q}  {J} ^{(k)}$ and the duality gap arising from the linearization approximation is given by \cite{jaggi2013revisiting}: 
\begin{equation}
    \begin{aligned}
   &~~~~  {J}^{(k)}-{J}\left(P^{\star},\lambda^{\star},q^{\star}\right)\\
      &\leq  \left\langle P^{(k)}- P^{\star},  \nabla_P J^{(k)}\right \rangle\\
     &~~~~~~~~~~~ +(\lambda^{(k)}-\lambda^{\star})\nabla_{\lambda} {J} ^{(k)}+\left(q^{(k)}-{q}^{\star}\right)^\top \nabla_{q}  {J} ^{(k)}\\
           &\leq  \max_{\left(\Phi, \tilde{P},\tilde{\lambda}, \tilde q\right) \in \mathcal{S}} \left\langle P^{(k)}- \tilde{P},  \nabla_P J^{(k)}\right \rangle\\
     &~~~~~~~~~~~ +(\lambda^{(k)}-\tilde{\lambda})\nabla_{\lambda} {J} ^{(k)}+\left(q^{(k)}-\tilde{q}\right)^\top \nabla_{q}  {J} ^{(k)}\\
  & =\left\langle  P^{(k)},  \nabla_P J^{(k)}\right \rangle+{\lambda ^{(k)}}\nabla_{\lambda} {J} ^{(k)}+\left({q ^{(k)}}\right)^\top \nabla_{q}  {J} ^{(k)}-\tilde{J}^{(k)},
\end{aligned} \label{eq_FW_gap}
\end{equation}
where the first inequality is due to the convexity of $ {J}\left(P,\lambda,q \right)$. Recalling the $1$-degree homogeneity of $ {J}\left(P,\lambda,q \right)$ in \eqref{eq_homogeneity}, it follows from Euler's homogeneous function theorem that $ {J}\left(P,\lambda,q \right)$ has unit total elasticity \cite[Definition 3.1]{shafieezadeh2018wasserstein}, satisfying
\begin{equation}
\left\langle  P^{(k)}, \nabla_P J^{(k)}\right \rangle+{\lambda ^{(k)}}\nabla_{\lambda} {J} ^{(k)}+\left({q ^{(k)}}\right)^\top \nabla_{q}  {J} ^{(k)}= J^{(k)}. \label{eq_UTE}
\end{equation}
Substituting \eqref{eq_UTE} into \eqref{eq_FW_gap} yields $\tilde{J}^{(k)} \leq{J}(P^{\star},\lambda^{\star},q^{\star})$. That is, the optimal value $\tilde{J}^{(k)}$ of the subproblem \eqref{eq_find_direction} always produces a lower bound for the optimal value $J(P^{\star},\lambda^{\star},q^{\star})$ of \eqref{eq_main_dual}, while the current feasible objective value ${J}^{(k)}=\hat{J}\left(P^{(k)},\lambda^{(k)}\right)$ serves as an upper bound. Consequently, the resulting gap between the lower and upper bounds yields a criterion for evaluating the convergence of the Frank-Wolfe algorithm. The implementation details of the tailored solution algorithm are summarized in Algorithm \ref{alg1}. Its convergence property is established as follows.

\begin{algorithm}
\caption{Frank-Wolfe Procedure for Solving Sinkhorn DRSE Problem \eqref{eq_main_dual}.} \label{alg1}
\begin{algorithmic}[1] 
\REQUIRE  Initial feasible solution $(\Phi^{(0)},P^{(0)},\lambda^{(0)},q  ^{(0)})$, reference covariance $\Sigma$, data samples $\{\hat\xi_i\}_{i=1}^N$, regularization parameter $\epsilon$, ball radius $\theta$ and optimality tolerance $\bar\gamma$. 
\STATE Initialize $\bar {P}$, $\underline {P}$, $\bar {\lambda}$, $\underline {\lambda}$, $\bar q$ and $\underline{q}$ based on \eqref{eq_lemma_upper_lambda_star} and \eqref{lemma_upper_else_bound_1}-\eqref{lemma_upper_else_bound_4} and $k\leftarrow 0$.
\WHILE{true}
\STATE Compute the gradient $\nabla_{P} J^{(k)}$, $\nabla_{\lambda} J^{(k)}$ and the step size $\rho^{(k)} = 2/(k+2)$.
\STATE Solve \eqref{eq_find_direction} and obtain the optimal value $\tilde{J}^{(k)}$ and the optimal solution $(\tilde{P}^{\star},\tilde{\lambda}^{\star},\tilde{\Phi}^{\star})$. 
\STATE Compute the optimality gap $\gamma^{(k)} \leftarrow \hat{J}\left(P^{(k)},\lambda^{(k)}\right)-\tilde{J}^{(k)}$. 
\IF{$ \gamma^{(k)}\leq\bar\gamma$}
\STATE \textbf{break}
\ELSE
\STATE Update $P^{(k+1)}\leftarrow (1-\rho^{(k)})P^{(k)}+\rho^{(k)}\tilde{P}^{\star}$ and $\lambda^{(k+1)}\leftarrow (1-\rho^{(k)})\lambda^{(k)}+\rho^{(k)}\tilde{\lambda}^{\star}$.
\STATE Iterate $k \leftarrow k+1$.
\ENDIF
\ENDWHILE 
\STATE Obtain a feasible solution by solving $ \min_{\Phi}  0~~
{\rm s.t.}~  \eqref{eq_main_dual_C2} ,~\eqref{eq_main_dual_C4},~\eqref{eq_main_dual_C5}$ with $P=P^{(k)}$ to obtain $\Phi^{\star}$.
\STATE Partition $ \Phi^{\star}$ into $\Phi^{\star}_x \in \mathbb{R}^{{(T+1)}n_x\times (T+1)n_x}$ and $\Phi^{\star}_y\in \mathbb{R}^{{(T+1)}n_x\times {(T+1)}n_y}$, and compute the observer gain $L^{\star} = {(\Phi^{\star}_{x})}^{-1} \Phi^{\star}_{y}$.
\STATE \textbf{Return} $L^{\star}$.
\end{algorithmic}
\end{algorithm}

\begin{theorem} \label{theo_conv}
For any integer $k >0$, $J^{(k)}$ computed by the $k$th iteration of Algorithm \ref{alg1} satisfies the following optimality-gap bound:
\begin{equation}
J^{(k)} -  J\left(P^{\star},\lambda^{\star},q^{\star}\right) 
 \leq\frac{2  \bar{\mathcal{C}}_{{ J}}(1+\gamma_0)}{k+2} ,\label{eq_theo_conv_concl}
\end{equation}
where $\gamma_0$ is the solution accuracy of the direction-finding subproblem \eqref{eq_find_direction} and
\begin{equation*}
    \begin{aligned}
    &\bar{\mathcal{C}}_{{ J}}={\left(\left\lVert \bar P-\underline{P}\right\rVert^2_{F}+\left|\bar  \lambda-\underline{\lambda}\right|^2+\left\lVert \bar q-\underline{q}\right\rVert^2\right)} \\
&\qquad \qquad \qquad \qquad \cdot   \left( \frac{\bar \lambda \epsilon+2\epsilon \lVert \bar P \rVert_{F}+n_\xi\bar{\lambda}\epsilon\lVert \Omega \rVert_2^2}{2\kappa^2}+ \frac{n_\xi\epsilon}{2\underline{\lambda}} \right).
    \end{aligned} 
\end{equation*}
\end{theorem}

\textit{Initialization Strategy: } To effectively execute Algorithm \ref{alg1}, a feasible initial solution $(P^{(0)},\lambda^{(0)})$ to \eqref{eq_main_dual} is required. We note the Wasserstein DRSE problem \cite{hajar2025distributionally} is an SDP and can be solved effectively using off-the-shelf convex programming solvers. Its optimal solution $(P_{\rm W}^{\star},\lambda_{\rm W}^{\star})$ itself satisfies constraints \eqref{eq_main_dual_C2}, \eqref{eq_main_dual_C4} and \eqref{eq_main_dual_C5}. Moreover, since $\lambda_{\rm W}^{\star} \Omega  \succ \lambda_{\rm W}^{\star} I \succeq  P_{\rm W}^{\star}+\kappa I$, constraints \eqref{eq_main_dual_C1} and \eqref{eq_main_dual_C3} are also satisfied. Henceforth, the optimal solution to the Wasserstein DRSE problem itself yields a feasible $ (P^{(0)},\lambda^{(0)})$.
An alternative option is to design $L$ through pole placement on $A_t - L_{t| t}C_t$. Then, a feasible $\Phi^{(0)}=[\Phi^{(0)}_{x}~~\Phi^{(0)}_{y}]$ can be computed based on \eqref{eq_Phi1} and \eqref{eq_Phi2} to respect the block lower-triangular constraint \eqref{MHE_lower_triangle}. On this basis, a feasible $P$ is given by $P^{(0)} = 
Q^{\top}(\Phi^{(0)})^{\top} \Phi^{(0)} Q$, and a feasible $\lambda$ can be set as $\lambda^{(0)} = \bar\sigma (  \Omega^{-1/2}(  P^{(0)}+ \kappa I)\Omega^{-1/2})$ satisfying \eqref{eq_main_dual_C3}. 
\begin{remark}
A bisection-search method was proposed in \cite[Algorithm 3]{wang2025sinkhorn} to solve the dual Sinkhorn DRO problem. When applied to \eqref{eq_main_dual}, it first determines lower and upper bounds $(\underline{\lambda},~\bar \lambda)$ on $\lambda$ and then repeatedly fixes $\lambda$, solves the resulting optimization problem, and updates the search interval. Each bisection iteration requires solving two convex problems with log-determinant objectives and semi-definite constraints. By contrast, each iteration of Algorithm \ref{alg1} solves only one SDP with a linear objective. Therefore, the proposed Frank-Wolfe solution algorithm requires less computation per iteration and makes better use of gradient information and the Wasserstein DRSE solution used for initialization. Besides, the two methods have different stopping criteria, namely, the gap between the lower and upper objective bounds for Algorithm \ref{alg1} and the interval length of $\lambda$ for the bisection search. Importantly, the objective bounds offer an additional advantage of Algorithm \ref{alg1}, because their gap directly certifies the solution accuracy, whereas the length of the $\lambda$ interval alone does not provide such a certificate.
\end{remark}

\section{Simulation Case Study} \label{case}
In this section, a simulation case study is presented to verify the performance of the proposed DRSE algorithm. We consider the linear time-varying system \eqref{eq_statepace} with time-varying coefficient matrices \cite{shafieezadeh2018wasserstein,xu2024globalized}:
\begin{align*}
   & A_t = \begin{bmatrix}
        0.9802 & 0.0196+0.099\Delta_t\\
        0&0.9802
    \end{bmatrix}+ \Delta A_t,~  C_t =\begin{bmatrix}
        1 &-1
    \end{bmatrix},\\
   & B_tB_t^\top = \begin{bmatrix}
      1.9608 & 0.0195\\
      0.0195& 1.9605
    \end{bmatrix},~
  D_t D_t^\top= 1,~B_t D_t^\top=0,
\end{align*}
where $B_t$ and $D_t$ are time-invariant. The nominal time variation in $A_t$ is specified by \textit{a priori} known component with $\Delta_t=t-t_0$, whereas $\Delta A_t$ denotes the unknown mismatch between the ground truth and its known nominal model. In the simulation, we set the initial state estimation $\hat{x}(0)=0_{n_x}$. The disturbance vector $\xi$ is generated as
$\xi=M(\rho)\xi'$, where $\xi'$ follows the mixture distribution:
\[
\xi'\sim
0.15\mathcal{L}(0,0.3I)
+0.85\mathcal{U}\bigl([-0.05,0.05]^{n_\xi}\bigr).
\]
Here, $M(\rho)$ is the lower-triangular Cholesky factor of the Toeplitz correlation matrix $R(\rho)$, satisfying 
$R(\rho)=M(\rho)M^\top(\rho)$ with 
\begin{equation*}
R(\rho)=
\begin{bmatrix}
1 & \rho & \cdots & \rho^{n_\xi-1}\\
\rho & 1 & \cdots & \rho^{n_\xi-2}\\
\rho^2 & \rho & \cdots & \rho^{n_\xi-3}\\
\vdots & \vdots & \ddots & \vdots\\
\rho^{n_\xi-1} & \rho^{n_\xi-2} & \cdots & 1
\end{bmatrix},
\end{equation*}
where $\rho=0.9$ is the correlation parameter. The SDPs are solved using Mosek \cite{mosek} via the YALMIP modeling interface \cite{lofberg2004yalmip} on a desktop computer with an Intel Core i7-10700 CPU and $64$G RAM.

We compare the proposed Sinkhorn DRSE with the Wasserstein DRSE, the $\mathcal{H}_2$ estimator, and the $\mathcal{H}_{\infty}$ estimator. For each method, we independently solve the estimator-design problem $20$ times. Each estimator is evaluated using an independent test dataset including $10{,}000$ samples, and the average MSE over all the designs is reported. In each simulation, the uncertainty $\xi$ is drawn from the same underlying distribution as that in the historical samples $\{\hat{\xi}_i\}_{i=1}^{N}$. The model mismatch is specified as
\begin{equation*}
    \Delta A_t=a_{\rm mis}\cdot v(t)\cdot\begin{bmatrix}0 & 0.099\\0 & 0\end{bmatrix},
\end{equation*}
where $v(t)\sim\mathcal{U}(-1,1)$ and the fixed amplitude parameter $a_{\rm mis}$ controls the level of model mismatch.

For estimator design, we generate $20$ independent historical data sets, each containing $N=40$ disturbance samples $\{\hat{\xi}_i\}_{i=1}^{N}$. Within each trial, all methods are designed using the same historical data set. For the $\mathcal{H}_2$ estimator, the disturbance covariance matrix $\Sigma$ is estimated from $\{\hat{\xi}_i\}_{i=1}^{N}$. For the $\mathcal{H}_{\infty}$ estimator, the bounded disturbance set is defined as $\left\{\xi ~\left|~ \xi^\top {\Sigma^{-1}} \xi  \leq \max_{i\in \mathbb{N}_{1:N}} \hat{\xi}_i^\top {\Sigma^{-1}}\hat{\xi}_i\right. \right \}$. For both DRSE methods, the ambiguity-set radius is selected from $\theta\in[10^{-4},10^{-1}]$, while the regularization parameter of the Sinkhorn DRSE is selected from $\epsilon\in[10^{-6},10^{-3}]$. A grid search is performed to select the best-performing hyperparameter combinations.

For performance evaluation, all estimators obtained in each trial are tested using the same set of $10{,}000$ disturbance samples in Monte Carlo simulations. These testing samples are generated independently of the historical data from the same underlying disturbance distribution. The resulting MSEs of the four estimators over all independent designs are reported in Table \ref{tab_MSE} for different model mismatches $a_{\rm mis}$ and estimation horizons $T$. When there is no model mismatch, i.e., $a_{\rm mis}=0$, both DRSE methods outperform the $\mathcal{H}_2$ and $\mathcal{H}_{\infty}$ estimators. As $a_{\rm mis}$ increases, the performance differences become more pronounced, demonstrating the tolerance of the proposed method to model uncertainty. This is because the model mismatch can be equivalently represented as a distributional shift in $\xi$. Provided that the shifted distribution remains within the prescribed ambiguity set, the corresponding distributionally robust performance guarantee continues to hold, thereby explaining the robust performance of the two DRSE methods. In contrast, the $\mathcal{H}_2$ estimator is sensitive to model mismatch and covariance misspecification, whereas the $\mathcal{H}_{\infty}$ estimator can be overly conservative due to its protection against extreme disturbances. Between the two DRSE methods, the entropic regularization employed by the Sinkhorn DRSE produces a smoother worst-case distribution that more closely resembles the underlying true distribution. Therefore, with a properly selected $\epsilon$, the Sinkhorn DRSE alleviates the conservatism of the Wasserstein DRSE and thus shows better estimation performance.

Next, we solve Sinkhorn DRSE \eqref{eq_main_dual} and Wasserstein DRSE using varying sample sizes $N$. For each $N$, the hyperparameters of both DRSE methods are calibrated separately using the same independent procedure described above. The resulting estimation performance of the $20$ independent designs evaluated on $10,000$ Monte Carlo simulations is presented in Fig. \ref{fig_MSE_N}. It can be seen that the proposed Sinkhorn DRSE achieves a higher accuracy than the Wasserstein DRSE. The performance of both DRSE methods improves as $N$ increases, which is attributed to the availability of richer distributional information. An obvious performance gap between the Sinkhorn DRSE and the Wasserstein DRSE can be evidenced under low sample sizes and narrows as $N$ increases. This is because as $N$ grows the worst distribution of the Wasserstein DRSE is supported on more discrete points, thereby better resembling the underlying continuous true distribution.
To visualize the high-dimensional worst-case distribution $\tilde{\mathbb{P}}$ \eqref{eq_theo_worst} in Theorem \ref{theo_worst}, we project $\tilde{\mathbb{P}}$ onto the first dimension $\xi_1'$ and obtain a univariate  marginal distribution $\tilde{\mathbb{P}}_1$. We first partition $\left\{U,u_i \right\}$ in \eqref{theo_worst_para} into: 
\begin{equation*}
    U=\begin{bmatrix}
        U_{11} &U_{12}\\ U_{21}&U_{22}
    \end{bmatrix},~ u_{i} =\begin{bmatrix}
        u_{i,1}\\u_{i,2}
    \end{bmatrix} ,
\end{equation*}
where $U_{11},u_{i,1}\in\mathbb{R}$, $U_{12}^\top,U_{21},u_{i,2} \in\mathbb{R}^{n_\xi-1}$ and $U_{22}\in\mathbb{R}^{(n_\xi-1)\times(n_\xi-1)}$. Similar to the argument in Theorem \ref{theo_worst}, using Gaussian integral yields: 
\begin{equation}
    {\rm d}\tilde{\mathbb{P}}_1(\xi_1') = \frac{1}{N}\sum_{i=1}^N v'_i\cdot \exp\left\{ U'  (\xi_1')^2 + u_i'\xi_1'\right\} {\rm d}\nu({\xi_1'}), \label{eq_worst_1}
\end{equation}
where
\begin{align*}
    U' &= U_{11}-U_{12}U_{22}^{-1}U_{21},~ u_i' = u_{i,1}- U_{12}U_{22}^{-1}u_{i,2},  \\
    v_i'& =\frac{(2\pi)^{(n_\xi-1)/2}}{\sqrt{\det(-2U_{22})}} \cdot v_i \cdot \exp\left\{-\frac{1}{4}u_{i,2}^{\top}U_{22}^{-1}u_{i,2}\right\}.
\end{align*}
By setting $N=40$, $\theta=0.1$ and different values of $\epsilon\in\{~10^{-4},~10^{-3},10^{-2}\}$, the densities of the marginal worst-case distributions $\tilde{\mathbb{P}}_1$ in \eqref{eq_worst_1} are shown in Fig. \ref{fig_worst}. It can be observed that a larger $\epsilon$ leads to a smoother worst-case distribution, thereby showing the potential for reducing the over-conservatism of Wasserstein DRSE corresponding to $\epsilon=0$. 

\begin{table}[t]
\centering
\caption{Resulting MSEs (mean $\pm$ standard deviation) of four estimation methods under different estimation horizons $T$ and model-mismatch levels
$a_{\rm mis}$.}
\label{tab_MSE}
\small
\setlength{\tabcolsep}{4pt}
\renewcommand{\arraystretch}{1.1}
\begin{tabular}{lccc}
\hline
Method & $a_{\rm mis}=0$ & $a_{\rm mis}=3$ & $a_{\rm mis}=5$\\
\hline
\multicolumn{4}{c}{$T=5$}\\
\hline
Sinkhorn
& $6.13\pm1.11$
& $6.47\pm1.16$
& $7.06\pm1.24$\\
Wasserstein
& $6.45\pm1.13$
& $6.82\pm1.15$
& $7.48\pm1.24$\\
$\mathcal{H}_2$
& $8.57\pm4.63$
& $9.27\pm5.20$
& $10.56\pm6.46$\\
$\mathcal{H}_{\infty}$
& $12.20\pm8.38$
& $13.72\pm9.75$
& $16.53\pm12.52$\\
\hline
\multicolumn{4}{c}{$T=7$}\\
\hline
Sinkhorn    & $10.20\pm1.95$& $11.13\pm2.10$& $12.91\pm2.38$\\
Wasserstein & $10.44\pm2.02$& $11.42\pm2.16$& $13.33\pm2.43$\\
$\mathcal{H}_2$ & $18.78\pm9.06$& $21.01\pm10.09$& $25.27\pm12.32$\\
$\mathcal{H}_{\infty}$ & $24.15\pm11.22$& $28.45\pm14.00$& $37.10\pm20.68$\\
\hline
\multicolumn{4}{c}{$T=10$}\\
\hline
Sinkhorn
& $16.98\pm2.04$
& $18.90\pm2.27$
& $22.79\pm2.76$\\
Wasserstein
& $17.35\pm2.43$
& $19.34\pm2.66$
& $23.36\pm2.85$\\
$\mathcal{H}_2$
& $27.17\pm8.59$
& $33.00\pm12.39$
& $46.02\pm22.57$\\
$\mathcal{H}_{\infty}$
& $31.35\pm10.77$
& $39.75\pm16.88$
& $60.23\pm35.35$\\
\hline
\end{tabular}
\end{table}

\begin{figure}
    \centering
    \includegraphics[width=0.95\linewidth]{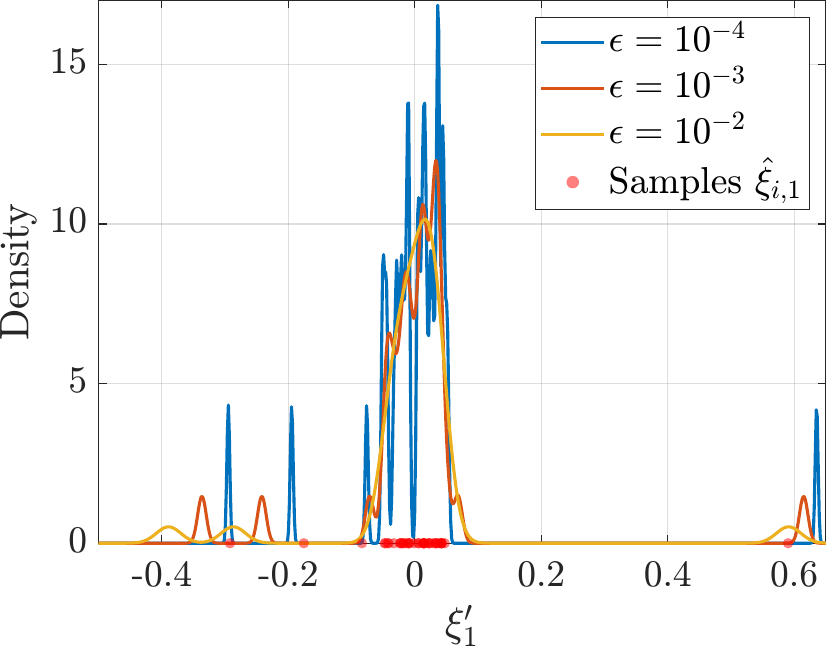}
    \caption{The first dimension $\hat\xi_{i,1}$ of the samples and the densities of the worst-case univariate marginal distribution $\tilde{\mathbb{P}}_1$ of \eqref{eq_main_dual} with $\theta=0.1$. 
}
    \label{fig_worst}
\end{figure}

\begin{figure}
    \centering
    \includegraphics[width=0.9\linewidth]{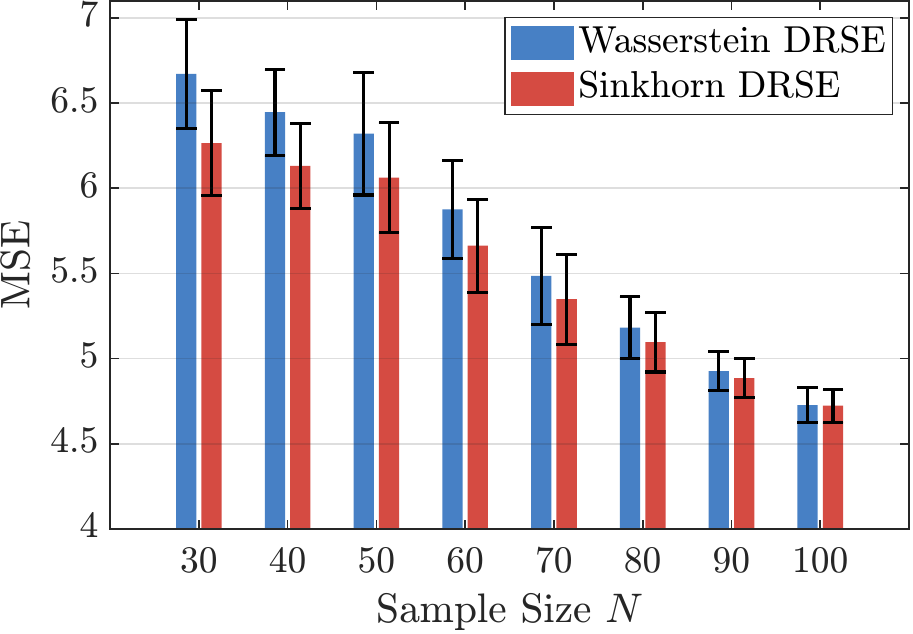}
    \caption{The average MSE of the Sinkhorn DRSE and the Wasserstein DRSE using different sample sizes $N$. The error bars are the standard error around the mean.}
    \label{fig_MSE_N}
\end{figure}

\subsubsection*{Performance under Gaussian uncertainty}
We further consider an ideal case under Gaussian disturbance and without model mismatch, i.e., $a_{\rm mis}=0$. The historical samples $\{\hat{\xi}_i\}_{i=1}^{N}$ and disturbances in the Monte Carlo simulations are independently drawn from the Gaussian distribution:
\begin{equation*}
\xi \sim \mathcal{N}\!\left(
0,\,
\operatorname{diag}\!\left(
0.1I_2,\,
I_T\otimes\operatorname{diag}(0.5,0.0001,0.03)
\right)
\right).
\end{equation*}

We additionally include an oracle $\mathcal{H}_2$ estimator designed using the true covariance, rather than the empirical covariance estimated from the $N$ historical samples. It therefore provides the theoretically optimal benchmark under the Gaussian assumption. All other estimators are designed using $\{\hat{\xi}_i\}_{i=1}^{N}$ with $N=40$. The hyperparameters of the two DRSE methods are calibrated using the same procedure described above. Table \ref{tab_gaussian_MSE} reports the mean and standard deviation of the resulting MSEs across the $20$ independent designs for different estimation horizons $T$. As expected, the oracle $\mathcal{H}_2$ estimator achieves the lowest MSE at every estimation horizon because of its theoretical optimality. The $\mathcal{H}_2$ estimator performs worse than both DRSE estimators even under Gaussian disturbances. This is because its empirical covariance is estimated in a small-sample regime, whereas the oracle $\mathcal{H}_2$ estimator uses the ground truth. Among the estimators constructed from only historical samples, the proposed Sinkhorn DRSE achieves the lowest MSE.

\begin{table}[t]
\centering
\caption{Resulting MSEs (mean $\pm$ standard deviation) of five
estimation methods under Gaussian disturbances with $N=40$ and
$a_{\rm mis}=0$.}
\label{tab_gaussian_MSE}
\small
\setlength{\tabcolsep}{3.5pt}
\renewcommand{\arraystretch}{1.1}
\begin{tabular}{lccc}
\hline
Method & $T=5$& $T=7$& $T=10$\\
\hline
Oracle $\mathcal{H}_2$ & $5.81\pm0.00$& $8.13\pm0.00$& $11.65\pm0.00$\\
$\mathcal{H}_2$ & $6.31\pm0.20$& $9.15\pm0.31$& $13.78\pm0.33$\\
$\mathcal{H}_{\infty}$ & $6.44\pm0.22$& $9.41\pm0.37$& $14.57\pm0.60$\\
Wasserstein& $6.22\pm0.16$& $9.05\pm0.26$& $13.67\pm0.37$\\
Sinkhorn& ${6.11\pm0.10}$& ${8.89\pm0.22}$& ${13.44\pm0.29}$\\
\hline
\end{tabular}
\end{table}

\subsubsection*{Computational aspects} Then, we compare the solution time of Algorithm \ref{alg1} with that of the bisection search in \cite[Algorithm 3]{wang2025sinkhorn}. In particular, Algorithm \ref{alg1} is initialized with the optimal Wasserstein DRSE solution obtained using the same radius $\theta$, while the bisection search uses the same multiplier bounds $(\underline{\lambda},\bar{\lambda})$ established in Lemma \ref{lemma_upper}. Since the two methods have different stopping criteria, we solve both methods to high accuracy for a fair comparison, and Fig. \ref{fig_conv_obj} shows the convergence trajectories. As shown in Fig. \ref{fig_conv_obj}, both methods converge to comparable values of the objective, whereas the proposed Algorithm \ref{alg1} requires less computation time, with this advantage becoming more pronounced for smaller $\epsilon$, e.g., $\epsilon=10^{-4}$ and $\epsilon=10^{-5}$. This computational advantage arises from the use of first-order information and the initialization provided by the solution of the Wasserstein DRSE. Although the time difference is modest at $\epsilon=10^{-3}$, Algorithm \ref{alg1} still certifies solution accuracy through its objective gap, whereas $\lambda$-interval criterion in the bisection search does not.
 In addition, Fig. \ref{fig_time} reports the average solution times of Algorithm \ref{alg1} and the bisection search with $\theta = 0.1$ and $\epsilon=10^{-4}$ for different sample sizes $N$. For a fair comparison, we set a stopping criterion $\bar \gamma=0.01$ for Algorithm \ref{alg1}, while the computation time of the bisection search is recorded when its objective value differs from the final objective value of Algorithm \ref{alg1} by no more than $\bar \gamma/2$. This allows the two methods to be compared at comparable solution accuracy. As shown in
Fig. \ref{fig_time},  our proposed Algorithm 1 consistently requires less computation time than the bisection search over all sample sizes $N$.
Moreover, the empirical convergence of our algorithm is found to be not sensitive to the selection of $\kappa$. A possible explanation is that the high-quality initialization provided by the Wasserstein DRSE keeps the iterates well inside the positive-definite region, making the constraint $\lambda\Omega-P\succeq\kappa I$ remain inactive in almost all iterations.
\begin{figure}[htbp]
    \centering
    \includegraphics[width=0.9\linewidth]{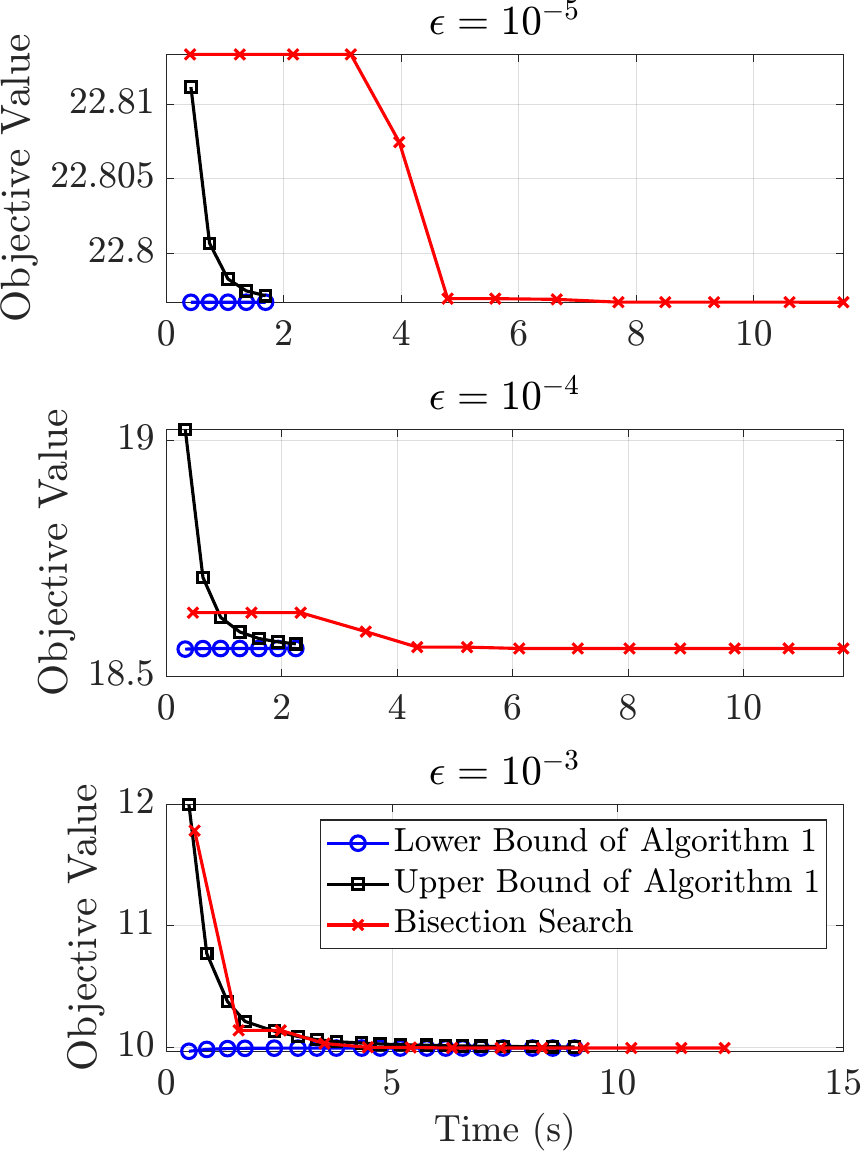}
    \caption{Convergence of the objective values obtained by Algorithm \ref{alg1} and the bisection search in solving \eqref{eq_main_dual}.}
    \label{fig_conv_obj}
\end{figure}

\begin{figure}
    \centering
    \includegraphics[width=0.9\linewidth]{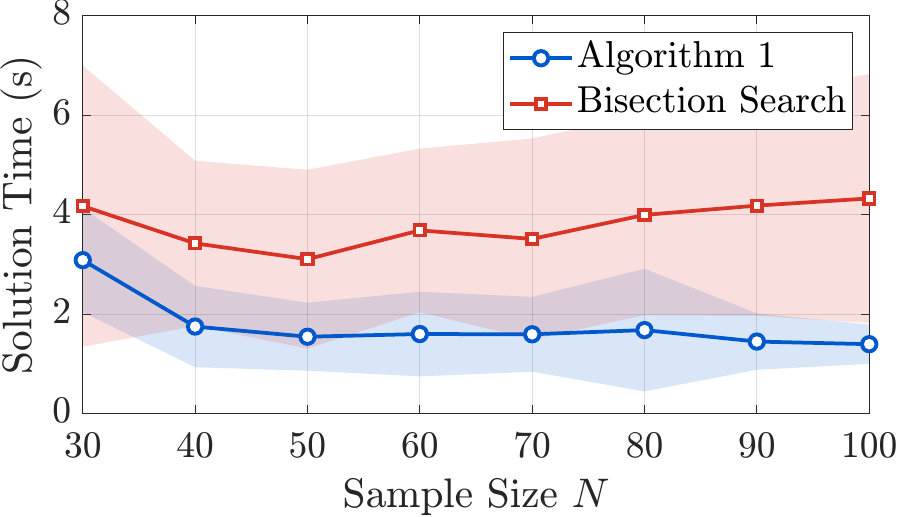}
    \caption{The average solution time of Algorithm \ref{alg1} and the bisection search using different sample sizes $N$. The shaded regions are the standard deviation around the mean.}
    \label{fig_time}
\end{figure}

\section{Conclusion} \label{concl}
In this work, a novel Sinkhorn DRSE method was put forward within the SLS framework, which seeks to alleviate the conservatism in Wasserstein DRSE by incorporating entropic regularization. We provided the first finite-sample probabilistic guarantee for the Sinkhorn ambiguity set. To tackle the resulting min-max design problem, we derived its exact dual reformulation as a finite-dimensional convex program. Then, by analyzing its limiting behaviors, we show that the Sinkhorn DRSE can be seen as an interpolation between $\mathcal{H}_2$ design and Wasserstein DRSE depending on different choices of hyperparameters. To solve the challenging problem, we identify a compact subset of its feasible set without compromising the global optimality and then apply the Frank-Wolfe procedure, for which the theoretical convergence was formally established. Finally, numerical examples demonstrate the superior performance of the Sinkhorn DRSE especially in small-sample regimes and the effective convergence of the tailored first-order solution algorithm. Future work will focus on hyperparameter tuning and extending our method to a recursive infinite-horizon case.

\section*{Appendix}
\renewcommand{\thesubsection}{\Alph{subsection}}

\subsection{Proof of Theorem \ref{theo_fin}} \label{app_1}
First, recall from Definition \ref{def_sink} that the Sinkhorn distance consists of a transportation cost and an entropic regularization term using the KL-divergence. Let $\pi$ be the joint distribution of the empirical distribution $\hat{\mathbb{P}}_N$ and the true distribution $\mathbb{P}^{\star}$. The KL-divergence can be reformulated as:
\begin{equation}
\begin{aligned}
   &~~~~ {\rm KL}\left( \pi| \alpha \odot\beta \right)= \mathbb{E}_{(\xi,\xi')\sim\pi} \left[ \log \frac{{\rm d}\pi(\xi,\xi') }{{\rm d} \alpha(\xi){\rm d} \beta(\xi')} \right]\\
     &= \mathbb{E}_{(\xi,\xi')\sim\pi} \left[ \log \frac{{\rm d}\pi(\xi,\xi') }{{\rm d} \hat{\mathbb{P}}_N(\xi){\rm d} \mathbb{P}^{\star}(\xi')}  + \log \frac{{\rm d}\mathbb{P}^{\star}(\xi')}{{\rm d} \beta(\xi')}  \right]\\
    &=\mathbb{E}_{(\xi,\xi')\sim\pi} \left[ \log \frac{{\rm d}\pi(\xi,\xi') }{{\rm d} \hat{\mathbb{P}}_N(\xi){\rm d} \mathbb{P}^{\star}(\xi')}  \right] +  \mathbb{E}_{\xi'\sim \mathbb{P}^{\star}}\left[\log \frac{{\rm d}\mathbb{P}^{\star}(\xi')}{{\rm d} \beta(\xi')}  \right]\\
     & = \mathbb{E}_{(\xi,\xi')\sim\pi} \left[ \log \frac{{\rm d} {\mathbb{P}}^{\star}(\xi'){\rm d}\pi_{\xi'}(\xi) }{{\rm d} \hat{\mathbb{P}}_N(\xi){\rm d} \mathbb{P}^{\star}(\xi')}  \right]+ {\rm KL}(\mathbb{P}^{\star} | \beta)\\
     & =  \underbrace{\mathbb{E}_{(\xi,\xi')\sim\pi} \left[ \log \frac{{\rm d}\pi_{\xi'}(\xi) }{{\rm d} \hat{\mathbb{P}}_N(\xi)}  \right]}_{\triangleq I(\xi,\xi') \text{ with } (\xi,\xi')\sim\pi } + {\rm KL}(\mathbb{P}^{\star} | \beta),
\end{aligned}\label{eq_theo_fin_sink1}
\end{equation}
 where $\pi_{\xi'}({\rm d}\xi)$ denotes the conditional distribution of $\xi$ given $\xi'$, and $I(\xi,\xi')$ is the mutual information between $\xi$ and $\xi'$ with $(\xi,\xi')\sim\pi$. Subsequently, $I(\xi,\xi')$ can be expressed as:
\begin{equation}
    \begin{aligned}
       I(\xi,\xi')
      & = H(\hat{\mathbb{P}}_N)-H(\xi |\xi')     \leq    \log N.
\end{aligned} \label{eq_theo_fin_I}
\end{equation}
where the first equation follows from the standard decomposition of $I(\xi,\xi')$ into the marginal entropy $H(\hat{\mathbb{P}}_N)$ and the conditional entropy $H(\xi|\xi' )$ with $(\xi,\xi')\sim\pi$  and the last inequality is obtained from $\hat{\mathbb{P}}_{N} = \frac{1}{N} \sum_{i=1}^{N} \delta_{\hat{\xi}_i}$ and $H(\xi|\xi' )\geq0$.
Next, substituting \eqref{eq_theo_fin_sink1} and \eqref{eq_theo_fin_I} into \eqref{eq_def_sink} yields the upper bound of the Sinkhorn distance: 
\begin{equation}
\begin{aligned}
&~~~~\mathcal{W}_{\epsilon}(\hat{\mathbb{P}}_N,{\mathbb{P}}^{\star}) \\
&\leq\inf_{\pi\in \Pi(\hat{\mathbb{P}}_N,{\mathbb{P}}^{\star})}\left\{\mathbb{E}_{(\xi,\xi')\sim\pi} \left[ c(\xi,\xi')  \right]+\epsilon [\log N+{\rm KL}(\mathbb{P}^{\star} | \beta)]
\right\} \\
&= {\inf_{\pi\in \Pi(\hat{\mathbb{P}}_N,{\mathbb{P}}^{\star})}\left\{\mathbb{E}_{(\xi,\xi')\sim\pi} \left[ c(\xi,\xi')  \right]\right\}}+\epsilon [\log N+{\rm KL}(\mathbb{P}^{\star} |\beta)]\\
&={\mathcal{W}(\hat{\mathbb{P}}_N,{\mathbb{P}}^{\star})}+\epsilon [\log N+{\rm KL}(\mathbb{P}^{\star} | \beta)].
\end{aligned}\label{eq_theo_fin_sink2}
\end{equation}
Under Assumption \ref{assu_light}, for any integer $k >0$ there exists a constant $a' > 0$ such that $a'\cdot\lVert \xi \rVert^k <  \exp\{ \lVert\xi\rVert^a\}$ holds for sufficiently large $\lVert \xi \rVert$, implying that the $k$th moment is finite. This allows us to write its mean and covariance as $\mathbb{E}_{\xi \sim \mathbb{P}^{\star}}[\xi] = \mu^{\star}$ and
$\mathbb{E}_{\xi \sim \mathbb{P}^{\star}}[(\xi-\mu^{\star})(\xi-\mu^{\star})^\top] = \Sigma^{\star}$. By virtue of the Pythagorean relation \cite[Theorem 2.8]{ay2017information}, we have: 
\begin{equation}
    {\rm KL}(\mathbb{P}^{\star} |\beta)
    =  {\rm KL}(\beta^{\star} |\beta)+{\rm KL}(\mathbb{P}^{\star} |\beta^{\star}), \label{eq_theo_fin_KL1}
\end{equation}
where $\mathcal{G}$ stands for the Gaussian family, and $\beta^{\star} = \arg\min_{\bar\beta \in \mathcal{G}} {\rm KL}(\mathbb{P}^{\star} | \bar\beta)$ can be seen as a projection of $\mathbb{P}^{\star}$ onto $\mathcal{G}$.
Since $\mathcal{G}$ is a subset of the exponential family, $\beta^{\star}$ can be uniquely determined by matching the sufficient statistics of $\mathbb{P}^{\star}$ corresponding to the first two moments, i.e., $\beta^{\star}=\mathcal{N}(\mu^{\star}, \Sigma^{\star})$. Furthermore, the KL-divergence ${\rm KL}(\beta^{\star} |\beta)$ between two Gaussian distributions admits a closed-form expression \cite[Appendix A.5]{seeger2004gaussian}:
\begin{equation}
\begin{split}
    & ~~~ {\rm KL}(\beta^{\star}|\beta) \\
    &    = \frac{1}{2} \left[(\mu^{\star})^\top \Sigma^{-1} {\mu^{\star}}-\log\det\left(\Sigma^{-1} \Sigma^{\star} \right) + {\rm Tr}\left\{\Sigma^{-1} \Sigma^{\star} \right\}-n_\xi\right]. 
\end{split}
\label{eq_theo_fin_KL2}
\end{equation}
Regarding the remaining term, Assumption \ref{assu_nodirac} implies a lower bound $H(\mathbb{P}^{\star}) \geq -\log b$, and thus we have: 
\begin{equation}
         {\rm KL}(\mathbb{P}^{\star}|\beta^{\star}) 
 =H(\mathbb{P}^{\star}, \beta^{\star})  -H(\mathbb{P}^{\star})\leq H(\mathbb{P}^{\star}, \beta^{\star})  +\log b,  \label{eq_theo_fin_KL3}
\end{equation}
where the cross-entropy $H(\mathbb{P}^{\star}, \beta^{\star})$ is given by:
\begin{equation}
    \begin{aligned}
&~~~~H(\mathbb{P}^{\star}, \beta^{\star})\\
&= - \mathbb{E}_{\xi \sim \mathbb{P}^{\star}}\left[ \log \left(\frac{{\rm d}\beta^{\star}(\xi)}{{\rm d} \nu(\xi)}\right) \right] \\
 &= - \mathbb{E}_{\xi \sim \mathbb{P}^{\star}}\left[ \log \left( \frac{\exp \left\{{-\frac{1}{2}(\xi - \mu^{\star})^\top ({\Sigma^{\star}})^{-1} (\xi - \mu^{\star}) }\right\}}
 {\sqrt{(2\pi)^{n_\xi} \det(\Sigma^{\star})}}  \right) \right] \\
 &=  \mathbb{E}_{\xi \sim \mathbb{P}^{\star}}\left[ \frac{1}{2}(\xi - \mu^{\star})^\top (\Sigma^{\star})^{-1} (\xi - \mu^{\star})\right]\\
 &~~~~~\qquad \qquad \quad+ \mathbb{E}_{\xi \sim \mathbb{P}^{\star}}\left[\frac{1}{2}\log \det(\Sigma^{\star}) + \frac{n_\xi}{2}\log(2\pi) \right] \\
  &=\frac{n_\xi}{2}\log(2\pi) + \frac{1}{2}\log \det(\Sigma^{\star}) + \frac{n_\xi}{2} .
\end{aligned} \label{eq_theo_fin_KL4}
\end{equation}
Thus, combining \eqref{eq_theo_fin_KL1}-\eqref{eq_theo_fin_KL4} yields $ {\rm KL}(\mathbb{P}^{\star} |\beta) \leq m-\log N$. Then, substituting this result into \eqref{eq_theo_fin_sink2} establishes an upper bound on the Sinkhorn distance as:
\begin{equation*}
{ \mathcal{W}}_{\epsilon}(\hat{\mathbb{P}}_N,{\mathbb{P}}^{\star})  \leq   \mathcal{W}(\hat{\mathbb{P}}_N,{\mathbb{P}}^{\star}) + \epsilon m.
\end{equation*}
By applying the finite-sample concentration inequality for Wasserstein balls
\cite[Theorem 2]{fournier2015rate}, we distinguish the following two cases
according to the exponent $a$ in Assumption \ref{assu_light}:
\begin{itemize}
    \item \textbf{Case 1:} $0<a<2$: \begin{equation}
    \begin{aligned}
&    \mathbb{P}^N\left\{{\mathcal{W}}_{\epsilon}\left( \hat{\mathbb{P}}_N,\mathbb{P}^{\star} \right) \geq \theta +\epsilon m\right\}\\
&\qquad\qquad\qquad\quad\leq \begin{cases}
 d_1(\theta)+    d_2(\theta) , & \text{if } \theta\le 1\\
 d_3(\theta), & \text{if } \theta> 1
\end{cases}
\end{aligned} \label{eq_theo_fin_bound1}
\end{equation}
    \item \textbf{Case 2:} $a>2$: \begin{equation}
    \begin{aligned}
&    \mathbb{P}^N\left\{{\mathcal{W}}_{\epsilon}\left( \hat{\mathbb{P}}_N,\mathbb{P}^{\star} \right) \geq \theta +\epsilon m\right\}  \\&\qquad\qquad\qquad\quad\leq \begin{cases}
d_1(\theta), & \text{if } \theta\le 1\\
d_4(\theta), & \text{if } \theta> 1 
\end{cases}
\end{aligned} \label{eq_theo_fin_bound2}
\end{equation}
\end{itemize}
where 
\begin{align*}
&d_1(\theta) = c_1 \exp\left\{ -c_2 N \theta ^{\max\{n_{\xi}/2,2\}} \right\},\\
&d_2(\theta) = c_1 \exp\left\{-c_2 \left(N \theta\right)^{(a-\bar a)/2}  \right\}, \\
&d_3(\theta) =  c_1 \exp\left\{-c_2(N\theta)^{a/2} \right\},~
d_4(\theta) = c_1 \exp\left\{  -c_2 N \theta ^{a/2} \right\}.
\end{align*}
With the bounds \eqref{eq_theo_fin_bound1} and \eqref{eq_theo_fin_bound2} at hand, we proceed to derive
an explicit expression for the radius $\theta_N(\eta)$ corresponding
to the prescribed confidence level $\eta$. For the case $0<a<2$, equating $d_3(\theta)=\eta$ and solving for $\theta=\theta_N(\eta)$, we obtain the solution $\theta_N(\eta)=\frac{1}{N}\left(\frac{\log(c_1/\eta)}{c_2}\right)^{2/a}$. Since $d_3(\theta)$ applies only for $\theta_N(\eta)>1$, this solution is adopted if and only if $N<\left(\frac{\log(c_1/\eta)}{c_2}\right)^{2/a}$. When $N\geq\left(\frac{\log(c_1/\eta)}{c_2}
\right)^{2/a}$, however, the equation $d_1(\theta)+d_2(\theta)=\eta$ generally does not admit a closed-form solution. We therefore choose an arbitrary $\eta'\in(0,\eta)$ and split the probability $\eta$ into $\eta'$ and
$\eta-\eta'$ for the two terms, respectively. Specifically, we solve
$d_1(\theta)=\eta'$ and $d_2(\theta)=\eta-\eta'$ separately and take the larger of the resulting radii $\theta_N(\eta,\eta')=\max\left\{\theta_1(\eta,\eta'),\theta_2(\eta,\eta')\right\}$ with $\theta_1(\eta,\eta')=\frac{1}{N}\left( \frac{\log\left( c_1 /\left(\eta-\eta'\right)\right)}{c_2}\right)^{{2}/{(a-\bar a)}}$ and $\theta_2(\eta,\eta')=\left( \frac{\log\left( c_1 /\eta'\right)}{c_2N}\right)^{{1}/{\max\{n_\xi/2,2\}}} $. If $\theta_N(\eta,\eta')\leq1$, \eqref{eq_theo_fin_bound} follows straightforwardly from $d_1\left(\theta_N(\eta,\eta')\right)+d_2\left(\theta_N(\eta,\eta')\right)\leq\eta$. 
If $\theta_N(\eta,\eta')>1$, incorporating the condition $N\geq\left(\frac{\log(c_1/\eta)}{c_2}\right)^{2/a}$ and the monotonicity of $d_3(\theta)$, we have
\begin{equation*}
\begin{aligned}
&d_3\left(\max\left\{\theta_1(\eta,\eta'),\theta_2(\eta,\eta')\right\}\right)\\
\leq& d_3(1)\leq d_3\left(\frac{1}{N}\left(\frac{\log(c_1/\eta)}{c_2}\right)^{2/a}\right)=\eta.
\end{aligned}
\end{equation*}
Thus, the radius $\theta_N(\eta)$ specified in \eqref{eq_theo_fin_thetaN1} guarantees \eqref{eq_theo_fin_bound}, which completes the proof for the case
$0<a<2$. We then turn to the simpler case
$a>2$, for which \eqref{eq_theo_fin_thetaN2} follows from analogous arguments, and this completes the proof.

\subsection{Proof of Theorem \ref{theo_dual}}\label{app_2}
Thanks to the similarity structure of the Sinkhorn DRO problem, the proof can be made by adapting \cite[Proposition 1]{cescon2025data} to the present setup. By using \cite[Theorem 1]{wang2025sinkhorn}, the inner maximization subproblem $\sup_{\mathbb{P}\in\mathcal{D}_{\epsilon}(\hat {\mathbb{P}}_N,\theta)} ~\mathbb{E}_{ \xi\sim  \mathbb{P}}\left[ \xi^\top P(\Phi)\xi  \right]$ in the primal problem \eqref{eq_primal} is feasible if and only if  
\begin{equation} 
\theta+\epsilon \mathbb{E}_{\xi \sim \hat {\mathbb{P}}_N } \left[\log \mathbb{E}_{\zeta  \sim \beta } \left[ e^{{-\lVert \xi-\zeta \rVert_2^2
}/{ \epsilon}}\right]\right]\geq 0, \label{eq_wang_fea}
\end{equation}
and admits the strong dual program: 
\begin{equation}
\inf_{\lambda \geq 0} \lambda \theta+\lambda\epsilon \mathbb{E}_{\xi \sim \hat {\mathbb{P}}_N } \left[\log \mathbb{E}_{\zeta  \sim \beta } \left[ e^{{( \zeta^\top P(\Phi )\zeta-\lambda\lVert \xi-\zeta \rVert_2^2
)}/{\lambda \epsilon}}\right]\right]. \label{eq_wang_dual}
\end{equation}
We start with characterizing the feasibility condition \eqref{eq_wang_fea}. Due to the empirical distribution $\hat{\mathbb{P}}_{N} = \frac{1}{N} \sum_{i=1}^{N} \delta_{\hat{\xi}_i}$, the expectation term in \eqref{eq_wang_fea} can be expressed as:
\begin{equation}
\begin{aligned}
  &  \mathbb{E}_{\xi \sim \hat {\mathbb{P}}_N } \left[\log \mathbb{E}_{\zeta  \sim \beta } \left[ e^{{-\lVert\xi-\zeta\rVert_2^2}/{ \epsilon}}\right]\right]\\
    =&\frac{1}{N}\sum_{i=1}^N \log  \mathbb{E}_{\zeta  \sim \beta } \left[ e^{{-\lVert \hat\xi_i-\zeta\rVert_2^2}/{ \epsilon}}\right],
\end{aligned}\label{eq_wang_fea_emprical} 
\end{equation}
where each logarithmic term can be rewritten as:
\begin{align*}
    &\log  \mathbb{E}_{\zeta  \sim \beta } \left[ e^{{-\lVert \hat\xi_i-\zeta\rVert_2^2}/{ \epsilon}}\right]
= \log  \Bigg[ \frac{1}{\sqrt{(2\pi)^{n_\xi} {\rm det}(\Sigma)}}\\
&\int_{\Xi}\exp\left\{\zeta^\top \frac{-2 I- \epsilon\Sigma^{-1}}{2\epsilon}\zeta 
+\frac{2\hat{\xi}_i^\top}{\epsilon} \zeta
-\frac{ \hat{\xi}_i^\top \hat{\xi}_i}{\epsilon}\right\}\Bigg]{\rm d}\nu(\zeta) \\  
= & \log  \left(\sqrt{\frac{( \epsilon/2)^{n_{\xi}}}{{ \det(\Sigma ) \det(\Omega)}}}
 \exp\left\{\frac{ \hat{\xi}_i^\top \Omega^{-1}\hat{\xi}_i}{ \epsilon}
-\frac{\hat{\xi}_i^\top \hat{\xi}_i}{\epsilon}
\right\}\right)\\
=&\frac{n_\xi}{2}\log\left(\frac{\epsilon}{2}\right)
-\frac{1}{2}\log\det(\Sigma)
-\frac{1}{2}\log\det(\Omega)\\
&\qquad\qquad\qquad\qquad\qquad\qquad\quad+\frac{1}{\epsilon}
\hat{\xi}_i^\top(\Omega^{-1}-I)\hat{\xi}_i.
\end{align*}
Here, the first equality follows from substituting the density of the Gaussian measure $\beta$ in \eqref{eq_Gaussian}, while the second equality follows from the Gaussian integral. This integral is finite if and only if $\Omega \succ 0$. Substituting the above expression successively into \eqref{eq_wang_fea_emprical} and \eqref{eq_wang_fea} yields the feasibility condition \eqref{eq_theo_dual_fea}.

We next derive the finite-dimensional reformulation of the dual problem \eqref{eq_wang_dual}. Following a procedure analogous to that used to reformulate \eqref{eq_wang_fea}, 
\begin{align*}
    &\lambda \theta+\lambda \epsilon\mathbb{E}_{\xi \sim \hat {\mathbb{P}}_N } \left[\log \mathbb{E}_{\zeta  \sim \beta } \left[ e^{{( \zeta^\top P(\Phi )\zeta-\lambda\lVert \xi-\zeta \rVert_2^2
)}/{\lambda \epsilon}}\right]\right]\\
= &\lambda \theta-\frac{\lambda\epsilon}{2}\log\det(\Sigma)
+\frac{\lambda\epsilon n_\xi}{2}
 \log\left(\frac{\lambda\epsilon}{2}\right) 
\\
&-\frac{\lambda\epsilon}{2}
 \log\det\left(\lambda\Omega-P(\Phi)\right) \\
&+\frac{1}{N}\sum_{i=1}^N
\left[ \lambda^2\hat\xi_i^\top \left(\lambda\Omega- P(\Phi)\right)^{-1}\hat\xi_i-\lambda\hat\xi_i^\top\hat\xi_i\right]
\end{align*}
is finite if and only if $\lambda\Omega-P(\Phi)\succ 0$. To obtain an epigraph representation of the inverse-quadratic terms, we introduce auxiliary variables $q_i$ satisfying $q_i\geq\lambda^2\hat\xi_i^\top
(\lambda\Omega-P)^{-1}\hat\xi_i-\lambda\hat\xi_i^\top\hat\xi_i$. Under $\lambda\Omega-P\succ0$, applying the Schur complement yields the constraint \eqref{eq_main_dual_C1}. Then, the equality constraint $P(\Phi)=Q^\top \Phi ^\top\Phi Q$ is nonconvex, so we thus replace it with its convex relaxation:
 \begin{equation}
     P\succeq P(\Phi)\triangleq Q^\top \Phi ^\top\Phi Q, \label{eq_relax}
 \end{equation}
which we next show to be exact.  Consider any feasible solution $(\bar\Phi ,\bar P,\bar \lambda,\bar q)$ of the relaxed problem satisfying constraints \eqref{eq_main_dual_C1}, \eqref{eq_main_dual_C4}, \eqref{eq_main_dual_C5}, $\lambda\Omega-P\succ0$ and \eqref{eq_relax}.  Due to $\bar P\succeq P(\bar\Phi)$, replacing $\bar P$ with $P(\bar\Phi)$ preserves \eqref{eq_main_dual_C1}. Since $\log\det(\cdot)$ is monotone increasing over the positive-definite cone, we have $-\log\det(\bar\lambda\Omega- P(\bar\Phi))
\leq
-\log\det(\bar\lambda\Omega-\bar P)$. Therefore, the equality-constrained solution $(\bar\Phi , P(\bar\Phi),\bar \lambda,\bar q)$ is feasible
and has an objective value not larger than that of the relaxed solution $(\bar\Phi,\bar P,\bar\lambda,\bar q)$, establishing the exactness of relaxation \eqref{eq_relax}. Applying the Schur complement to \eqref{eq_relax} yields \eqref{eq_main_dual_C2}, and choosing any sufficiently small $ 0<\kappa\leq
\lambda_{\min}(\lambda^\star\Omega-P^\star)$ replaces the strict inequality $\lambda\Omega-P\succ0$, thereby yielding \eqref{eq_main_dual_C3}. Finally, combining the preceding
reformulated objective and constraints with the SLS
constraints \eqref{eq_primal_SLS1} and
\eqref{eq_primal_SLS2} completes the proof.

\subsection{Proof of Theorem \ref{theo_worst}} \label{app_3}      
According to \cite[Remark 4]{wang2025sinkhorn}, the worst-case distribution $ \tilde{\mathbb{P}}$ of \eqref{eq_primal} always has the form
\begin{equation}
\begin{aligned}
  {\rm d}\tilde{\mathbb{P}}(\xi')=&
\mathbb{E}_{\xi\sim\hat{\mathbb{P}}_N}\bigg[(\phi(\xi))^{-1}\\
&\cdot \exp \left\{ \frac{( \xi')^\top P(\Phi^\star) \xi'-\lambda^\star c(\xi,\xi')}{\lambda^\star \epsilon}\right\}\bigg]{\rm d}\beta(\xi') ,
\end{aligned} \label{eq_theo_worst_ref}
\end{equation}
where $\phi(\xi)=\mathbb{E}_{\zeta \sim \beta}\left[ \exp \left\{ ({ \zeta^\top P(\Phi^\star)\zeta -\lambda^\star\lVert\xi-\zeta\rVert^2})/{\lambda^\star \epsilon}\right\}\right]$ is a normalizing constant. Following the same Gaussian-integration argument as in the proof of Theorem \ref{theo_dual}, $\phi(\xi)$ can be derived as:
\begin{align*}
  &\phi(\xi) \\
  =& \int_{\Xi}\exp\left\{\zeta^\top \frac{P(\Phi^\star)-\lambda^\star I}{\lambda^\star \epsilon} \zeta + \frac{2}{\epsilon}\xi^\top\zeta- \frac{1}{\epsilon}\xi^\top\xi \right\}{\rm d}\beta(\zeta) \\
  =&\frac{1}{\sqrt{(2\pi)^{n_\xi} {\rm det}(\Sigma)}}   \int_{\Xi}\exp\left\{\zeta^\top U \zeta + \frac{2}{\epsilon}\xi^\top\zeta- \frac{1}{\epsilon} \xi^\top\xi \right\}{\rm d}\nu(\zeta) \\
    =& \frac{{(\lambda^\star\epsilon)}^{n_\xi/2}}{\sqrt{\det(2\Sigma(\lambda^\star \Omega-P\left( \Phi^\star\right)))}}  \exp\left\{-\frac{1}{\epsilon^2}\xi^\top U^{-1}\xi-\frac{1}{\epsilon}\xi^\top\xi\right\}.
   \end{align*}
Substituting $\phi(\xi)$ and $\hat{\mathbb{P}}_{N}$ into \eqref{eq_theo_worst_ref} yields \eqref{eq_theo_worst}, which completes the proof.

\subsection{Proof of Lemma \ref{lemma_upper}} \label{app_5}
The proof relies on constructing a linear lower bound for the objective $\hat J( P,\lambda)$ with respect to $\lambda$. To this end, we begin by bounding the nonlinear terms $\hat g(P,\lambda)$ and $h( P,\lambda)$, separately. Observing that $\lambda \Omega - P \preceq \lambda \Omega$ and invoking the monotonicity of the determinant over $\mathbb{S}^{n_\xi}_+$ yield a lower bound of $h(P,\lambda)$: 
\begin{equation}
    \begin{aligned}
h(P,\lambda)
  &  \geq   \frac{\lambda\epsilon}{2}
  \left[n_\xi \log\left(\frac{\lambda\epsilon}{2}\right)-\log \det\left(\lambda \Omega\right)  \right]\\
&=     \frac{\lambda\epsilon}{2}\left[n_\xi \log\left(\frac{\epsilon}{2}\right)-\log \det\left( \Omega\right)\right] \\&\triangleq\underline{h}(\lambda).
\end{aligned} \label{eq_lemma_upper_h_lower}
\end{equation}
The analysis in Theorem \ref{theo_dual} implies that $g\left(q^{\star}\right)=\hat g(P^{\star},\lambda^{\star})$ and ${ J(P^{\star}, \lambda^{\star},q^{\star})}={\hat J(P^{\star}, \lambda^{\star})} $ hold at the optimum. By analogous arguments, we derive a lower bound for $g\left(q^{\star}\right)$: 
\begin{equation}
    \begin{aligned}
g\left(q^{\star}\right)&=\hat g(P^{\star},\lambda^{\star})\\
       & \geq  \frac{1}{N} \sum_{i=1}^N {\left[\left(\lambda^{\star}\right)^2\hat{\xi}_i^\top \left(\lambda^{\star} \Omega\right)^{-1}\hat{\xi}_i -\lambda^{\star}  \hat{\xi}^\top_i \hat{\xi}_i\right]}\\
  &   =  \frac{\lambda^{\star}}{N} \sum_{i=1}^N {\hat{\xi}_i^\top \left(\Omega^{-1}-I\right)\hat{\xi}_i }  \triangleq \underline{{ g}}(\lambda^{\star}).
\end{aligned} \label{eq_lemma_upper_g_lower}
\end{equation}
Building upon \eqref{eq_lemma_upper_h_lower} and \eqref{eq_lemma_upper_g_lower}, the optimal value ${ J(P^{\star}, \lambda^{\star},q^{\star})}$ satisfies 
\begin{align*}
 &~~~{ J(P^{\star}, \lambda^{\star},q^{\star})}= \hat J(P^{\star},\lambda^{\star}) \\
 & \geq f(\lambda^{\star})+\underline{ h}(\lambda^{\star})+  \underline{{ g}}(\lambda^{\star})\\
   & =  \lambda^{\star} \bigg(\theta-\frac{\epsilon}{2}\log \det(\Sigma) + \frac{\epsilon n_\xi}{2}\log\left(\frac{\epsilon}{2}\right)\\
  & ~~~~~~~ -\frac{\epsilon}{2}\log \det(\Omega)+\frac{1}{N}\sum_{i=1}^N \hat{\xi}_i^\top \left(\Omega^{-1}-I\right)\hat{\xi}_i \bigg)\\
 & = \lambda^{\star} \left(\theta-\theta_0 \right),
\end{align*}
where $\theta_0$ was already defined in \eqref{eq_theo_dual_fea}. Due to the optimality of $\left(\lambda^{\star}, P^{\star}\right)$, it holds for any feasible $\left\{\lambda^{(0)}, P^{(0)}\right\}$ that
\begin{align}
    \lambda^{\star}\left(\theta-\theta_0 \right) \leq     \hat J(P^{\star},\lambda^{\star})  \leq     \hat J(P^{(0)},\lambda^{(0)}).  
\end{align}
Given the strict feasibility condition $\theta > \theta_0$ established in Theorem \ref{theo_dual}, one can obtain \eqref{eq_lemma_upper_lambda_star} and this completes the proof. 

\subsection{Proof of Lemma \ref{lemma_upper_else}} \label{app_6} 
We start by bounding the optimal solution $(P^\star,\lambda^\star)$. By applying the Schur complement to \eqref{eq_main_dual_C2}, we derive the lower bound $  P\succeq
Q^\top\Phi^\top\Phi Q  \succeq0= \underline{P}$, which implies $P^{\star} \succeq \underline{P}$. Conversely, substituting $\lambda^{\star} \leq \bar \lambda$ into \eqref{eq_main_dual_C3} yields the upper bound $P^{\star} \preceq  \lambda ^{\star}\Omega -\kappa I  \preceq \bar\lambda \Omega-\kappa I = \bar P$. Furthermore, substituting $P\succeq \underline{P}$ into \eqref{eq_main_dual_C3} implies $\lambda\Omega \succeq \underline{P}+\kappa I$, which leads to the lower bound $   \lambda\geq \bar \sigma \left(\Omega^{-1/2}  (\underline{P}+\kappa I) \Omega^{-1/2}\right) = \bar\sigma \left(  \kappa \Omega^{-1}\right) =\underline{\lambda} $ because $\Omega = I +\frac{ \epsilon}{2}\Sigma^{-1}\succ0$, and this implies $\lambda^{\star} \geq\underline{\lambda}$. Since the objective \eqref{eq_main_dual_O} is strictly increasing in each
$q_i$, the Schur complement constraint \eqref{eq_main_dual_C1} is active
at optimality. With $(P^{\star},\lambda^{\star})$ confined to the established bounded interval, we can bound the optimal solution ${q^{\star}_i}$ as follows: 
\begin{equation*}
\begin{split}
   q_i^{\star} &= \left(\lambda^{\star}\right)^2 \hat{\xi}_i^\top (\lambda^{\star} \Omega - P^{\star})^{-1} \hat{\xi}_i - \lambda^{\star} \hat{\xi}_i^\top \hat{\xi}_i  \\
    &\geq \left \{  \begin{split}
    \inf &~  \lambda^2 \hat{\xi}_i^\top (\lambda \Omega - P)^{-1} \hat{\xi}_i - \lambda \hat{\xi}_i^\top \hat{\xi}_i  \\
    {\rm s.t.} &~  \underline{ P} \preceq P\preceq \bar P,~ \underline{ \lambda} \leq \lambda  \leq \bar \lambda ,~ \lambda\Omega-P \succeq \kappa I
    \end{split} \right .\\
    &\geq \underline{\lambda}^2 \hat{\xi}_i^\top (\bar{\lambda} \Omega - \underline P)^{-1} \hat{\xi}_i - \bar{ \lambda} \hat{\xi}_i^\top \hat{\xi}_i,
\end{split}
\end{equation*} 
and $q_i^\star \leq \bar\lambda^2 \hat{\xi}_i^\top  \hat{\xi}_i/\kappa - \underline{ \lambda} \hat{\xi}_i^\top \hat{\xi}_i=\bar{q}_i$ is due to $\lambda \Omega-P \succeq \kappa I$ by using a similar argument. This completes the proof.

\subsection{Proof of Theorem \ref{theo_conv}} \label{app_7}
We employ the proof machinery of the convergence of the Frank-Wolfe algorithm, which hinges on finding an upper bound for the curvature constant $\mathcal{C}_{{ J}}$  \cite[Lemma 7]{jaggi2013revisiting}: 
\begin{equation}
   \mathcal{C}_{{J}} \leq \left[{\rm diam_{\lVert\cdot\rVert}\left( {\rm proj}_{P,\lambda,q}(\mathcal{S})\right)}\right]^2 {\rm Lip}\left(\nabla J\right). \label{eq_cJ_up}
\end{equation}
First, we aim at proving the upper bound of the feasible diameter ${\rm diam_{\lVert\cdot\rVert}({\rm proj}_{P,\lambda,q}({\mathcal{S}}))}$ over $
(P,\lambda,q)$ with respect to an arbitrary norm $\lVert \cdot \rVert$ and the Lipschitz constant  ${\rm Lip}(\nabla  {J})$ of the gradient. 
Building upon the established bounds for $(P,\lambda,q)$ in \eqref{eq_feasible_S}, we derive the diameter of ${\rm proj}_{P,\lambda,q}(\mathcal{S})$ with respect to the Frobenius norm:
\begin{equation}
\begin{aligned}
    &~~~~{\rm diam}_{\lVert\cdot\rVert_{F}}\left({\rm proj}_{P,\lambda,q}( \mathcal{S} )\right) \\
   & = \left \{ \begin{split} \sup_{P,P',\lambda,\lambda',q,q'} & ~\sqrt{\left\lVert P-P'\right\rVert^2_{ F}+\lvert\lambda -\lambda'\rvert^2+\lVert q -q'\rVert^2}\\
   {\rm s.t. } ~~~~& ~(P,\lambda,q)\in {\rm proj}_{P,\lambda,q}(\mathcal S)\\
   &~ (P',\lambda',q')\in  {\rm proj}_{P,\lambda,q}(\mathcal S)
   \end{split} \right .\\
&\leq  \sqrt{ \left\lVert \bar P-\underline{P}\right\rVert^2_{F}+\left|\bar  \lambda-\underline{\lambda}\right|^2+\left\lVert \bar q-\underline{q}\right\rVert^2},
\end{aligned}\label{eq_diam_upper}
\end{equation}
where the inequality follows from the monotonicity of $\lVert \cdot\rVert_{ F}$ over $\mathbb{S}^{n_\xi}_+$. Then, for any perturbation  $\Delta=(\Delta_{P},\Delta_{\lambda},\Delta_{q})\in \mathbb{S} ^{n_\xi}\times \mathbb{R}\times\mathbb{R}^N$ on $ (P,\lambda,q)$, ${J}(P+\Delta_{P},\lambda+\Delta_{\lambda},q+\Delta_q)=f(\lambda+\Delta_{\lambda})+g(q+\Delta_q)+h(P+\Delta_{P},\lambda+\Delta_{\lambda})$ holds, where the linear terms $f(\lambda+\Delta_{\lambda})$ and $g(q+\Delta_q)$ can be expressed as $f(\lambda+\Delta_{\lambda})=f(\lambda)+\Delta_{\lambda}\theta-{\Delta_{\lambda}\epsilon}\log \det(\Sigma)/{2}$ and $g(q+\Delta_q)=g(q)+ 1_N^{\top}\Delta_q/N$ , respectively. The nonlinear term $h(P+\Delta_{P},\lambda+\Delta_{\lambda})$ can be expanded as:
\begin{align*}
  &~~~~~  h(P+\Delta_{P},\lambda+\Delta_{\lambda})\\
    &= \frac{(\lambda+\Delta_{\lambda})\epsilon n_\xi}{2}\log\left(\frac{(\lambda+\Delta_{\lambda})\epsilon}{2}\right) \\
    &~~~~~ -\frac{(\lambda+\Delta_{\lambda})\epsilon}{2}\log \det\left((\lambda+\Delta_{\lambda}) \Omega-(P+\Delta_{P})\right)\\
   & = h(P,\lambda)+\frac{\epsilon n_\xi}{2}\left[\log\left(\frac{\lambda\epsilon}{2}\right) +1\right]\Delta_\lambda+\frac{\epsilon n_\xi}{4\lambda}\Delta^2_\lambda+\mathcal{O}(\Delta^3_\lambda)\\
    &~~~~~ +  \left[ -\frac{\epsilon}{2}\log\det\left( \lambda\Omega - P\right) - \frac{\lambda\epsilon}{2}{\rm Tr}\left\{\left( \lambda\Omega - P\right)^{-1}\Omega\right\} \right] \Delta_\lambda\\
       &~~~~~  + \frac{\lambda\epsilon}{2}{\rm Tr}\left\{  \left( \lambda\Omega - P\right)^{-1} \Delta_P \right\}\\
       &~~~~~ - \frac{\epsilon}{2} \Delta_\lambda {\rm Tr}\left\{ \left( \lambda\Omega - P\right)^{-1}(\Delta_\lambda \Omega - \Delta_P) \right\} \\
       &~~~~~ + \frac{\lambda\epsilon}{4}{\rm Tr}\left\{ \left(\left(\lambda\Omega - P\right)^{-1}(\Delta_\lambda \Omega - \Delta_P) \right)^2 \right\}   + \mathcal{O}(\|\Delta\|^3).
\end{align*}
Combining these yields the overall expansion of the objective function:
\begin{align*}
    &~~~~~{J}(P+\Delta_{P},\lambda+\Delta_{\lambda},q+\Delta_q)\\
    &=
 J(P, \lambda,q) + \left\langle G, \Delta_P  \right\rangle + g^\top \begin{bmatrix}
    \Delta_\lambda^\top & \Delta_q^\top
 \end{bmatrix}^\top\\
 &\qquad\qquad   +     \frac{1}{2} \begin{bmatrix} \Delta_\lambda \\  \rm vec(\Delta_P)\\\Delta_q \end{bmatrix}^\top H \begin{bmatrix} \Delta_\lambda \\ {\rm vec}\left\{\Delta_P\right\} \\
 \Delta_q  \end{bmatrix} + \mathcal{O}(\|\Delta\|^3),
\end{align*}
where the gradient components $G = \nabla_P J\in \mathbb{S}_+^{n_\xi}$, $g= \begin{bmatrix}
         \nabla_{\lambda} J^\top&\nabla_{q} J^\top
    \end{bmatrix}^\top \in \mathbb{R}^{N+1}$, and the Hessian matrix $H \in \mathbb{S}^{n_\xi^2+1+N}$ is given by:
\begin{equation*}
    H = \begin{bmatrix}
   H_{\lambda\lambda}   & H_{P\lambda } ^\top &0\\
        H_{P\lambda} &   H_{PP}&0 \\
        0 &0 &0
    \end{bmatrix} , 
\end{equation*}
with the blocks defined as follows: 
    \begin{align}
    &    H_{PP} = \frac{\lambda\epsilon}{2} \left[ \left(\lambda\Omega - P\right)^{-1} \otimes \left(\lambda\Omega - P\right)^{-1} \right], \label{eq_H_block1}\\
   &  H_{P\lambda}=-\frac{\epsilon}{2} {\rm vec}\left\{ (\lambda\Omega-P)^{-1} P \left(\lambda\Omega - P\right)^{-1} \right\}, \label{eq_H_block2}\\
  \begin{split}
      &      H_{\lambda\lambda} = \frac{n_\xi\epsilon}{2\lambda} - \epsilon{\rm Tr}\left\{\left(\lambda\Omega - P\right)^{-1}\Omega\right\}\\
 &\qquad\quad  +  \frac{\lambda\epsilon}{2}{\rm Tr}\left\{\left(\lambda\Omega - P\right)^{-1}\Omega \left(\lambda\Omega - P\right)^{-1}\Omega\right\}.
    \end{split}\label{eq_H_block3}
\end{align}
The Lipschitz constant of $\nabla {J}$ can be upper-bounded by the spectral norm of $H$. We proceed with bounding the spectral norms of the Hessian blocks \eqref{eq_H_block1}-\eqref{eq_H_block3}:
\begin{align*}
    \lVert      H_{P P} \rVert_2  &\leq \frac{\lambda\epsilon}{2}   \left\lVert \left(\lambda \Omega-P \right)^{-1}\right\rVert^2_2 \leq   \frac{\bar \lambda \epsilon}{2\kappa^2}, \\
   \lVert      H_{P \lambda} \rVert_2 &=    \frac{\epsilon}{2} \left\lVert \left(\lambda \Omega-P \right)^{-1} P \left(\lambda \Omega-P \right)^{-1}\right\rVert_{F} \\
   &\leq  \frac{\epsilon}{2} \left\lVert \left(\lambda \Omega-P \right)^{-1}\right\rVert^2_2 \lVert P \rVert_{F}   \leq   \frac{\epsilon \lVert \bar P \rVert_{ F}}{2\kappa^2},   \\
        \lVert   H_{\lambda\lambda}\rVert_2&\leq \frac{n_\xi\epsilon}{2\lambda}  +  \frac{\lambda\epsilon}{2}{\rm Tr}\left\{\left(\lambda \Omega-P \right)^{-1}\Omega \left(\lambda \Omega-P \right)^{-1}\Omega\right\} \\
&\leq \frac{n_\xi\epsilon}{2\lambda}  + \frac{n_\xi\lambda\epsilon}{2} \left \lVert \left(\lambda \Omega-P \right)^{-1}\right\rVert^2_2 \left\lVert \Omega \right\rVert_2^2\\
&\leq \frac{n_\xi\epsilon}{2\underline{\lambda}} +\frac{n_\xi\bar{\lambda}\epsilon \lVert \Omega \rVert_2^2}{2\kappa^2}  ,
\end{align*}
which are derived from the triangle inequality building upon the established bounds $P\preceq \bar P$, $\lambda \leq \bar\lambda$ and $(\lambda \Omega-P )^{-1}\preceq  I/\kappa  $. Invoking the triangle inequality of spectral norms, the Lipschitz bound satisfies: 
\begin{equation}
    \begin{aligned}
     {\rm Lip}\left(\nabla J\right) & \leq \lVert H \rVert_2 
     \leq \lVert    H_{PP}\rVert_2+2\lVert    H_{P\lambda}\rVert_2+\lVert    H_{\lambda\lambda}\rVert_2\\
     & \leq \frac{\bar \lambda \epsilon+2\epsilon \lVert \bar P \rVert_{F}+n_\xi\bar{\lambda}\epsilon\lVert \Omega \rVert_2^2}{2\kappa^2}+ \frac{n_\xi\epsilon}{2\underline{\lambda}} . 
\end{aligned}\label{eq_lip_upper}
\end{equation}
 Substituting \eqref{eq_diam_upper} and \eqref{eq_lip_upper} into \eqref{eq_cJ_up}, the curvature bound  $\mathcal{C}_{{J}} \leq \bar{ \mathcal{C}}_{{J}}$ holds. Consequently, by invoking \cite[Theorem 1]{jaggi2013revisiting}, the optimality gap of the $k$th iteration can be bounded as $J^{(k)}- J(P^{\star},\lambda^{\star},q^{\star})\leq{2\mathcal{C}_{{J}}}(1+\gamma_0)/{(k+2)}\leq{2 \bar{\mathcal{C}}_{{J}}}(1+\gamma_0)/({k+2})$, which completes the proof. 

\section*{References}

\bibliographystyle{IEEEtran}
\bibliography{ref}

\begin{IEEEbiography}[{\includegraphics[width=1in,height=1.25in,clip,keepaspectratio]{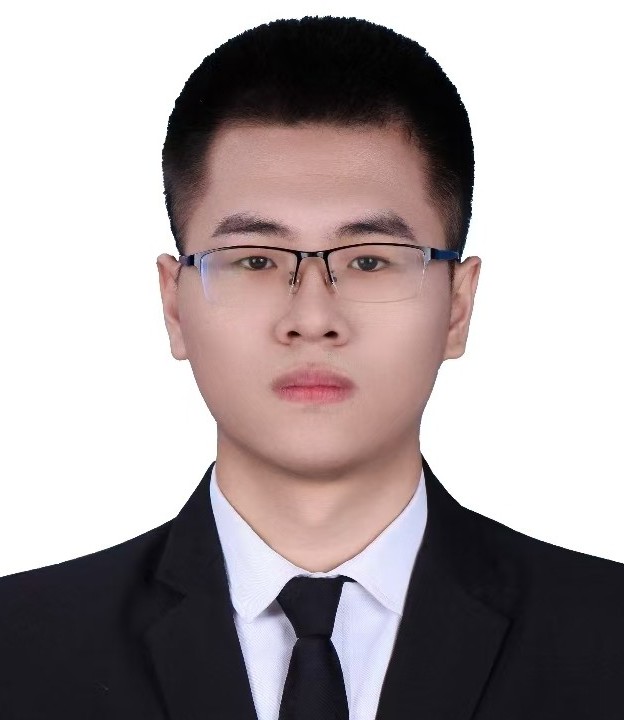}}]{Yulin Feng} received the B.S. degree in automation from Northeastern University, Shenyang, China, in 2023. He is currently pursuing the Ph.D. degree in control science and engineering from Tsinghua University, Beijing, China. His current research interests include data-driven predictive control, fault diagnosis and state estimation. He was a recipient of the First Prize of Fang Chong-Zhi Best Paper Award in 2025.
\end{IEEEbiography}

\begin{IEEEbiography}[{\includegraphics[width=1in,height=1.25in,clip,keepaspectratio]{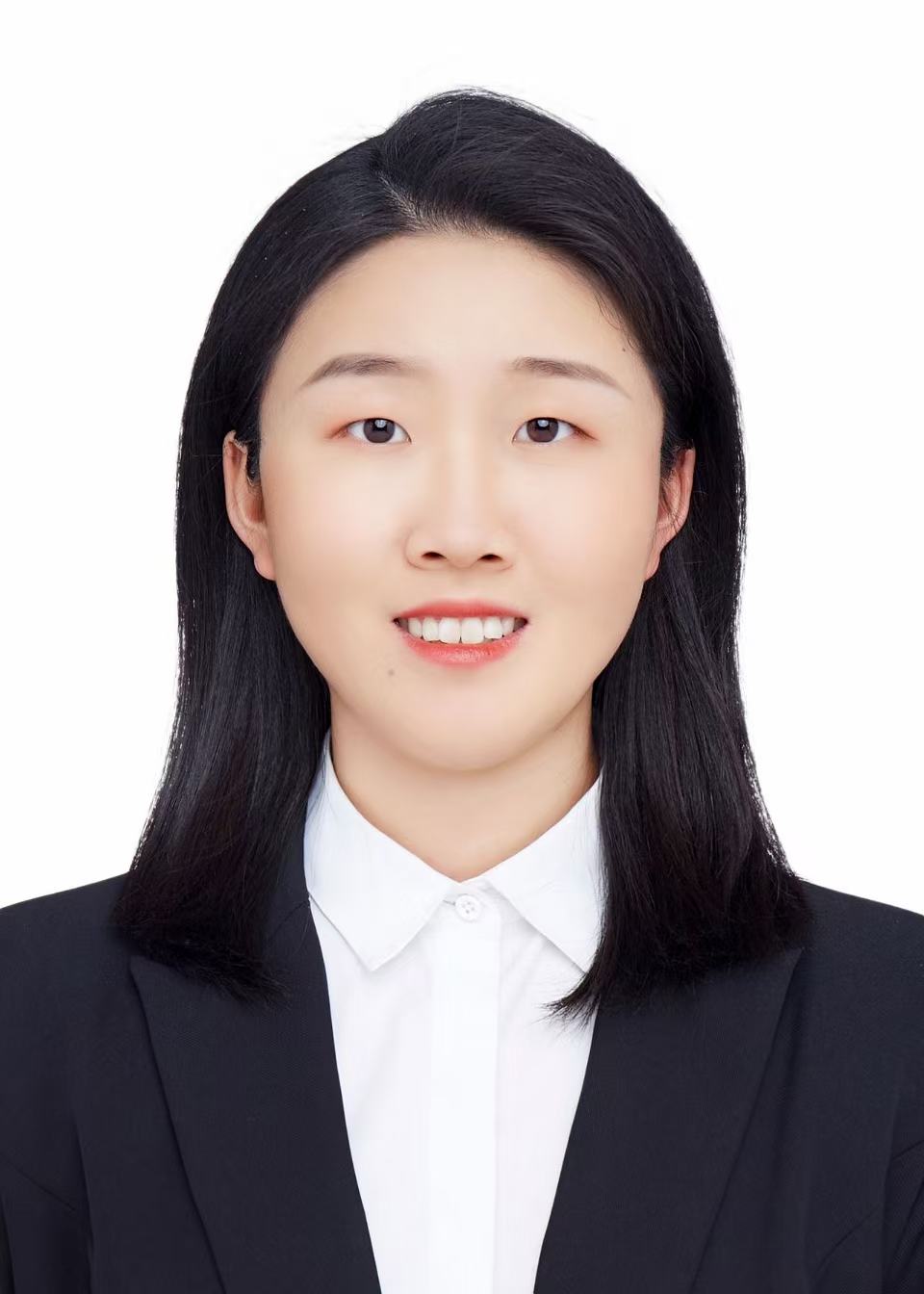}}]{Xianyu Li} received the B.S. degree in automation from Tsinghua University, Beijing, China, in 2022, where she is currently pursuing the Ph.D. degree in control science and engineering. Her current research interests include optimization under uncertainty and contextual optimization.
\end{IEEEbiography}

\begin{IEEEbiography}[{\includegraphics[width=1in,height=1.25in,clip,keepaspectratio]{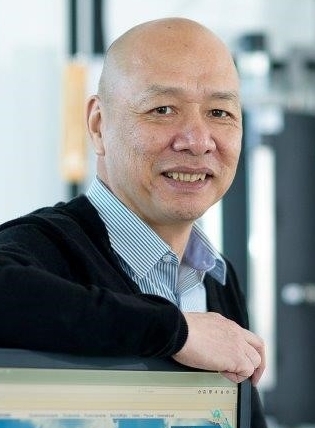}}]{Steven X. Ding} received Dr.-Ing. degree in electrical engineering from the Gerhard-Mercator University of Duisburg, Germany, in 1992. From 1992 to 1994, he was a R\&D engineer at Rheinmetall GmbH, Germany. From 1995 to 2001, he was a full-professor of control engineering at the University of Applied Science Lausitz in Senftenberg, Germany, and served as a vice president of this university during 1998--2000. Since 2001, he has been a full-professor of control engineering and the head of the Institute for Automatic Control and Complex Systems (AKS) at the University of Duisburg-Essen. His research interests are model-based and data-driven fault diagnosis, fault tolerant systems and their applications in industry with a focus on automotive systems, chemical processes and renewable energy systems.
\end{IEEEbiography}

\begin{IEEEbiography}[{\includegraphics[width=1in,height=1.25in,clip,keepaspectratio]{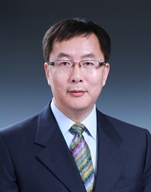}}]{Hao Ye} received the bachelors and Ph.D. degrees in automation from Tsinghua University, Beijing, China, in 1992 and 1996, respectively. He has been with the Department of Automation,
Tsinghua University, since 1996, where he is currently a Professor. He is mainly interested in fault detection and diagnosis of dynamic systems.
\end{IEEEbiography}

\begin{IEEEbiography}[{\includegraphics[width=1in,height=1.25in,clip,keepaspectratio]{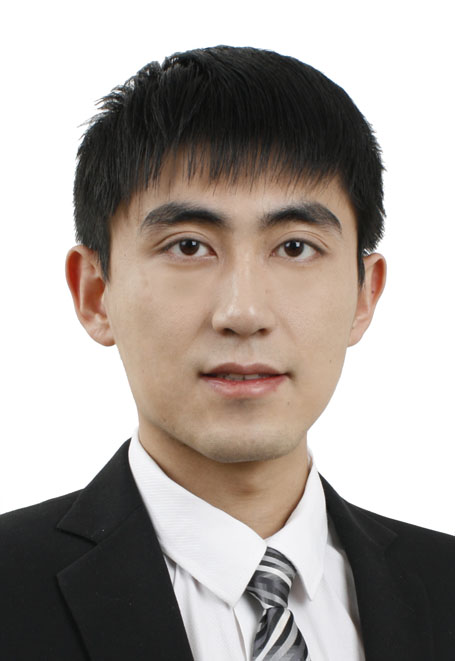}}]{Chao Shang} (Member, IEEE) received the B.Eng. degree in automation and the Ph.D. degree in control science and engineering from Tsinghua University, Beijing, China, in 2011 and 2016, respectively. After working as a Postdoctoral Fellow at Cornell University, he joined the Department of Automation, Tsinghua University in 2018, where he is currently a tenured Associate Professor. His research interests range over data-driven modeling, monitoring, control and optimization with applications to industrial manufacturing processes.

Prof. Shang is the recipient of Springer Excellent Doctorate Theses Award, Emerging Leaders in Control Engineering Practice, Best Paper Award of 1st International Conference on Industrial Artificial Intelligence, Zijing Scholarship, and Teaching Achievement Award from Tsinghua University.
\end{IEEEbiography}

\end{document}